\documentclass{imsart}

\RequirePackage[OT1]{fontenc}
\RequirePackage{amsthm,amsmath}
\RequirePackage{natbib}
\RequirePackage[colorlinks,citecolor=blue,urlcolor=blue]{hyperref}

\usepackage{mathrsfs}
\usepackage{graphicx}
\usepackage{amssymb}
\usepackage{dsfont}
\usepackage{a4}
\usepackage{url}

\startlocaldefs
\numberwithin{equation}{section}
\theoremstyle{plain}
\newtheorem{theorem}{Theorem}[section]
\newtheorem{corollary}{Corollary}[section]
\newtheorem{lemma}{Lemma}[section]
\newtheorem{remark}{Remark}[section]

\def\dd{{\rm {d}}}

\def\diag{{\rm {diag}}}

\def\var{\mathsf{var}}
\def\mp{\partial}
\def\tr{\mbox{\rm trace}}
\def\Ex{\mathsf{E}}

\def\1b{{\mathbf 1}}
\def\ra{\rightarrow}

\def\TT{^\top}

\def\ml{\lambda}
\def\ma{\alpha}

\def\mve{\varepsilon}
\def\mt{\theta}
\def\mth{\hat\theta}
\def\mg{\gamma}
\def\ms{\sigma}

\def\0b{{\mathbf 0}}

\def\SA{{\mathscr A}}
\def\SB{{\mathscr B}}
\def\SE{{\mathscr E}}

\def\SM{{\mathscr M}}
\def\SN{{\mathscr N}}
\def\SP{{\mathscr P}}

\def\SS{{\mathscr S}}
\def\ST{\mathbb{T}}
\def\SV{{\mathscr V}}

\def\SX{{\mathscr X}}
\newcommand{\vsp}{\vspace{0.3cm}}
\newcommand{\fin} {\hfill $\Box$ }
\endlocaldefs

\begin{document}

\begin{frontmatter}

\title{Extended generalised variances, with applications}

\runtitle{Extended generalised variances}

\begin{aug}

\author{\fnms{Luc} \snm{Pronzato}\ead[label=e1]{pronzato@i3s.unice.fr}},
\author{\fnms{Henry P.} \snm{Wynn}\ead[label=e2]{H.Wynn@lse.ac.uk}}
\and
\author{\fnms{Anatoly A.} \snm{Zhigljavsky}\ead[label=e3]{ZhigljavskyAA@cf.ac.uk}}

\runauthor{L. Pronzato et al.}

\affiliation{CNRS, London School of Economics and Cardiff University}

\address{Laboratoire I3S, UMR 7172, UNS, CNRS; 2000, route des Lucioles, Les Algorithmes, b\^at. Euclide B, 06900 Sophia Antipolis, France.
\printead{e1}}

\address{London School of Economics, Houghton Street, London, WC2A 2AE, UK.
\printead{e2}}

\address{School of Mathematics, Cardiff University, Senghennydd Road, Cardiff, CF24 4YH, UK.
\printead{e3}}

\end{aug}

\begin{abstract} We consider a measure $\psi_k$ of dispersion which extends the notion of Wilk's generalised variance for a $d$-dimensional distribution, and is based on the mean squared volume of simplices of dimension $k\leq d$ formed by $k+1$ independent copies. We show how $\psi_k$ can be expressed in terms of the eigenvalues of the covariance matrix of the distribution, also when a $n$-point sample is used for its estimation, and prove its concavity when raised at a suitable power. Some properties of dispersion-maximising distributions are derived, including a necessary and sufficient condition for optimality. Finally, we show how this measure of dispersion can be used for the design of optimal experiments, with equivalence to $A$ and $D$-optimal design for $k=1$ and $k=d$ respectively. Simple illustrative examples are presented.
\end{abstract}



\end{frontmatter}

\section{Introduction}
The idea of dispersion is fundamental to statistics and with different terminology, such as potential, diversity, entropy, information and capacity, stretches over a wide area. The variance and standard deviation are the most prevalent for a univariate distribution,
and Wilks generalised variance is the term usually reserved for the determinant of the covariance matrix, $V$, of a multivariate distribution.  Many other measures of dispersion have been introduced and a rich area comprises those that are
order-preserving with respect to a dispersion
ordering; see \cite{shaked1982dispersive, oja1983descriptive, giovagnoli1995multivariate}. These are sometimes referred to as {\em measures of  peakness} and  {\em peakness ordering}, and are related to the large literature on dispersion measures which grew out of the
Gini coefficient, used to measure income inequality \citep{gini1921} and diversity in biology, see \cite{rao1982a}, which we will discuss briefly below.

In the definitions there are typically two kinds of dispersion, those measuring some kind of mean distance, or squared distance, from a central value, such as in the usual definition of variance, and  those  based on the expected distance, or squared distance, between two independent copies from the same distribution, such as the Gini coefficient. It is this second type that will concern us here and we will generalise the idea in several ways by replacing distance by volumes of simplices formed by $k$ independent copies and by transforming the distance, both inside the expectation and outside. This use of volumes makes our measures of dispersion sensitive to the dimension of the subspace where the bulk of the data lives in.

The area of optimal experimental design is another which has provided a  range of dispersion measures. Good designs, it is suggested, are those whose parameter estimates have low dispersion. Typically, this means that the design measure, the spread of the observation sites, {\em  maximises} a measure of dispersion and we shall study this problem.

We think of a dispersion measure as a  functional directly on the distribution. The basic functional
is an integral, such  as a moment. The property we shall stress for such functionals most is concavity: that a functional does not decrease under mixing of the distributions. A fundamental theorem in Bayesian learning is that we expect concave functionals to decrease through taking of observations, see Section~\ref{S:learning} below.

Our central result (Section~\ref{S:squared volume}) is an identity for the mean squared volume of simplices of dimension $k$, formed by $k+1$ independent copies, in terms of the eigenvalues of the covariance matrices or equivalently in terms of sums of the determinants of $k$-marginal covariance matrices. Second, we note that after an appropriate (exterior) power transformation the functional becomes concave. We can thus ($i$) derive properties of measures that maximise this functional (Section~\ref{S:maxent}), ($ii$) use this functional to measure the dispersion of parameter estimates in regression problems, and hence design optimal experiments which minimise this measure of dispersion (Section~\ref{S:design}).

\section{Dispersion measures}

\subsection{Concave and  homogeneous functionals}\label{S:intro}
Let  $\SX$ be a compact subset of $\mathds{R}^d$, $\SM$ be the set of all probability measures on the Borel subsets of $\SX$
and  $\phi: \SM \longrightarrow \mathds{R}^+$ be a functional defined on $\SM$. We will be interested in the functionals $\phi(\cdot)$ that are (see Appendix for precise definitions)

\begin{itemize}
  \item[(a)] shift-invariant,
  \item[(b)] positively homogeneous of a given degree $q$,
   and
    \item[(c)] concave: $\phi[(1-\ma)\mu_1+\ma\mu_2] \geq (1-\ma)\phi(\mu_1)+\ma \phi(\mu_2)$ for any $\ma\in(0,1)$ and any two measures $\mu_1$, $\mu_2$ in $\SM$.
 \end{itemize}

For $d=1$, a common example of a functional satisfying the above properties, with $q=2$ in (b), is the variance
$$
\sigma^2(\mu) = E^{(2)}_\mu - E^2_\mu = \frac12 \int \int (x_1-x_2)^2\, \mu(\dd x_1)\,\mu(\dd x_2)\, ,
$$
where $E_\mu=\Ex(x)=\int x\,\mu(\dd x)$ and $E^{(2)}_\mu= \int x^2\,\mu(\dd x)$. Concavity follows from linearity of $E^{(2)}_\mu$, that is, $E^{(2)}_{(1-\ma)\mu_1+\ma\mu_2} = (1-\ma)E^{(2)}_{\mu_1} + \ma E^{(2)}_{\mu_2}$, and Jensen's inequality which implies $E^{2}_{(1-\ma)\mu_1+\ma\mu_2} \leq  (1-\ma)E^{2}_{\mu_1} + \ma E^{2}_{\mu_2}$.

Any moment of $\mu \in \SM$ is a homogeneous functional of a suitable degree. However, the variance is the only moment which satisfies (a) and (c). Indeed, the shift-invariance implies that the moment should be central, but the variance is the only concave functional among the central moments, see Appendix.
In this sense, one of the aims of this paper is a generalisation of the concept of variance.

In the general case $d\geq 1$, the double variance $2\sigma^2(\mu)$  generalises to
\begin{equation}
\label{variance}
  \phi(\mu) =  \int\int \|x_1-x_2\|^2  \,\mu(\dd x_1)\, \mu(\dd x_2) =  2\int \|x-E_\mu\|^2\, \mu(\dd x) =  2\,\tr(V_\mu) \,,
\end{equation}
where $\|\cdot\|$ is the $L_2$-norm in $\mathds{R}^d$ and  $V_\mu$  is the covariance matrix of $\mu$.
This functional, like the variance, satisfies conditions (a)-(c) with $q=2$.

The functional \eqref{variance} is the double integral of the squared distance between two random points distributed according to the measure $\mu$.
%
Our main interest will be concentrated around the general class of functionals defined by
\begin{equation}\label{dV}
  \phi(\mu)= \phi_{[k],\delta,\tau}(\mu) = \left(\int \ldots \int \SV_k^\delta(x_1,\ldots,x_{k+1})\,\mu(\dd x_1)\ldots \mu(\dd x_{k+1}) \right)^\tau\,, \ k\geq 2
\end{equation}
for some $\delta$ and $\tau$ in $\mathds{R}^+$, where $\SV_k(x_1,\ldots,x_{k+1})$ is the volume of the $k$-dimensional simplex (its area when $k=2$) formed by the $k+1$ vertices $x_1,\ldots,x_{k+1}$ in $\mathds{R}^d$, with $k=d$ as a special case.
Property (a) for the functionals \eqref{dV} is then a straightforward consequence of the shift-invariance of $\SV_k$, and positive homogeneity of degree $q=k\,\delta\tau$ directly follows from the positive homogeneity of $\SV_k$ with degree $k$. Concavity will be proved to hold for $\delta=2$ and $\tau\leq 1/k$ in Section~\ref{S:squared volume}. There, we show that this case  can be considered as a natural extension of \eqref{variance} (which corresponds to $k=1$), with $\phi_{[k],2,\tau}(\mu)$ being expressed as a function of $V_\mu$, the covariance matrix of $\mu$. The concavity for $k=\tau=1$ and all $0< \delta \leq 2$, follows from
the fact that $B(\lambda)=\lambda^\alpha$, $0 < \alpha \leq 1$, is a Bernstein function,
which will be discussed briefly below. The functionals \eqref{dV} with $\delta=2$ and $\tau>0$, $1 \leq k \leq d$, can be used to define a family of criteria for optimal experimental design, concave for $\tau\leq 1/k$, for which an equivalence theorem can be formulated.



\subsection{Quadratic entropy and learning}\label{S:learning}
In a series of papers \citep{rao1982a, rao1982b, rao1984convexity, rao2010quadratic} C.R. Rao and co-workers  have introduced a quadratic entropy which is a generalised version of the $k=2$ functional of this section but with a general kernel $K(x_1,x_2)$ in $\mathds{R}^d$:
\begin{equation}\label{quadratic-entropy}
  Q_R = \int \int K (x_1,x_2) \mu(\dd x_1)\mu(\dd x_2)\,.
\end{equation}
For the discrete version
$$
Q_R = \sum_{i=1}^N \sum_{j=1}^N K(x_{i},x_{j})\,p_i\,p_j,
$$
Rao and co-workers developed a version of the Analysis of Variance (ANOVA), which they called Anaysis of Quadratic Entropy (ANOQE), or Analysis of Diversity (ANODIV).
The Gini coefficient, also used in the continuous and discrete form is a special case with $d=1$ and $K(x_1,x_2) = |x_1-x_2|$.

As pointed in \citep[Chap.~3]{rao1984convexity}, a necessary and sufficient condition for the functional $Q_R$ to be concave is
\begin{equation}
\label{semi}
\int \int K(x_1,x_2) \nu(\dd x_1)\nu(\dd x_2) \leq 0
\end{equation}
for all measures $\nu$ with $\int \nu(\dd x)=0$. The discrete version of this is
$$
\sum_{i=1}^N \sum_{j=1}^N  K (x_i,x_j)\, q_i\, q_j \leq 0
$$
for any choice of real numbers  $q_1,\ldots,q_N$ such that $\sum_{i=1}^N q_i=0$.
\citet{schilling2012bernstein} discuss
the general problem of finding for what class of continuous functions $B(\cdot)$ of  $\|x_1 - x_2\|^2$ does the kernel
$
K(x_1, x_2) = B\left (\|x_1 - x_2\|^2\right)
$
satisfy \eqref{semi}: the solution is that $B(\cdot)$ must be a so-called Bernstein function. We do not develop these ideas here, but note that $B(\lambda) = \lambda^{\alpha}$ is a Bernstein function for all $0 < \alpha \leq 1$.
This is the reason that, above,
we can claim concavity for $k=1$ and all $0 < \delta \leq 2$ in \eqref{dV}.

\cite{hainy2014learning}
discuss the link to embedding and  review some basic results related to Bayesian learning. One asks
what is the class of functionals $\psi$ on a distribution $\mu(\theta)$ of a parameter in the Bayesian statistical learning
such that for all $\mu(\theta)$ and all sampling distributions $\pi(x | \theta)$ one expects to learn, in the preposterior sense:
$\psi(\mu(\theta)) \leq \Ex_\nu \psi(\pi(\theta|X))$,  with $X\sim \nu$.
The condition is that $\psi$ is convex, a result which has a history but is usually attributed to
\cite{degroot1962uncertainty}.
This learning is enough to justify calling such a functional a generalised information functional, or a general learning functional. Shannon information falls in this class, and earlier versions of the result were for Shannon information. It follows that wherever, in this paper, we have a concave functional then its negative is a learning functional.

\section{Functionals based on squared volume}
\label{S:squared volume}


In the rest of the paper we focus our attention on the functional
$$
\mu\in\SM \longrightarrow \psi_k(\mu)=\phi_{[k],2,1}(\mu)=\Ex\{ \SV_k^2(x_1,\ldots,x_{k+1}) \}\,,
$$
which corresponds to the mean squared volume of simplices of dimension $k$ formed by $k+1$ independent samples from $\mu$. For instance,
\begin{equation}\label{psi2}
  \psi_2(\mu) = \int\int\int \SV_2^2(x_1,x_2,x_3)\, \mu(\dd x_1)\, \mu(\dd x_2)\, \mu(\dd x_3)\,,
\end{equation}
with $\SV_2(x_1,x_2,x_3)$ the area of the triangle formed by the three points with coordinates $x_1$, $x_2$ and $x_3$ in $\mathds{R}^d$, $d\geq 2$. Functionals $\phi_{[k],\delta,\tau}(\mu)$ for $\delta \neq 2$ will be considered in another paper, including the case of negative $\delta$ and $\tau$ in connection with space-filling design for computer experiments.

Theorem~\ref{Prop:1} of Section~\ref{S:main-theorem} indicates how $\psi_k(\mu)$ can be expressed as a function of $V_\mu$, the covariance matrix of $\mu$, and shows that $\phi_{[k],2,1/k}(\cdot)$ satisfies properties (a), (b) and (c) of Section~\ref{S:intro}. The special case of $k=d$
was known to Wilks (1932, 1960)\nocite{wilks1932, wilks1960} in his introduction of generalised variance, see also \cite{van1965note}. The connection with U-statistics is exploited in Section~\ref{S:empirical}, where an unbiased minimum-variance estimator of $\psi_k(\mu)$ based on a sample $x_1,\ldots,x_n$ is expressed in terms of the empirical covariance matrix of the sample.

\subsection{Expected squared $k$-simplex volume}\label{S:main-theorem}

\begin{theorem}\label{Prop:1} Let the $x_i$ be i.i.d.\ with the probability measure $\mu\in\SM$. Then, for any $k\in\{1,\ldots,d\}$, we have
\begin{eqnarray}
\psi_k(\mu) &=& \frac{k+1}{k!} \, \sum_{i_1<i_2<\cdots<i_k}  \det[\{V_\mu\}_{(i_1,\ldots,i_k)\times(i_1,\ldots,i_k)}] \label{Tha} \\
            &=& \frac{k+1}{k!} \, \sum_{i_1<i_2<\cdots<i_k}  \ml_{i_1}[V_\mu]\times \cdots \times \ml_{i_k}[V_\mu]\,, \label{Thb}
\end{eqnarray}
where $\ml_{i}[V_\mu]$ is the $i$-th eigenvalue of the covariance matrix $V_\mu$ and all $i_j$ belong to $\{1,\ldots,d\}$.
Moreover, the functional $\psi_k^{1/k}(\cdot)$ is shift-invariant, homogeneous of degree $2$ and concave on $\SM$.
\end{theorem}

The proof uses the following two lemmas, see Appendix.

%
%

\begin{lemma}\label{L:1}
Let the $k+1$ vectors $x_1,\ldots,x_{k+1}$ of $\mathds{R}^k$ be i.i.d.\ with the probability measure $\mu$, $k\geq 2$. For $i=1,\ldots,k+1$, denote
$z_i=(x_i\TT \ 1)\TT$. Then
$$
\Ex \left\{ \det\left[\sum_{i=1}^{k+1} z_i z_i\TT \right] \right\} = (k+1)!\, \det[V_\mu]\,.
$$
\end{lemma}

\begin{lemma}\label{L:2}
The matrix functional $\mu\mapsto V_\mu$ is Loewner-concave on $\SM$, in the sense that,  for any $\mu_1$, $\mu_2$ in $\SM$ and any $\ma\in(0,1)$, \begin{equation}\label{conc0}
V_{(1-\ma)\mu_1+\ma\mu_2} \succeq (1-\ma)V_{\mu_1}+\ma V_{\mu_2},
\end{equation}
 where $A \succeq B$ means that $A-B$ is nonnegative definite.
\end{lemma}

\noindent \emph{\textbf{Proof of Theorem \ref{Prop:1}.}}
When $k=1$, the results follow from $\psi_1(\mu)=2\, \tr(V_\mu)$, see \eqref{variance}.
Using Binet-Cauchy formula, see, e.g., \cite[vol.~1, p.~9]{Gantmacher66}, we obtain
\begin{eqnarray*}
\lefteqn{ \SV_k^2(x_1,\ldots,x_{k+1}) = } \\
&& \frac{1}{(k!)^2} \, \det \left(
                            \left[
                              \begin{array}{c}
                                (x_2-x_1)\TT \\
                                (x_3-x_1)\TT \\
                                \vdots \\
                                (x_{k+1}-x_1)\TT \\
                              \end{array}
                              \right]
                              \begin{array}{c}
                               \left[ (x_2-x_1) \ (x_3-x_1) \ \cdots (x_{k+1}-x_1) \right] \\
                                \mbox{} \\
                                \mbox{} \\
                                \mbox{} \\
                              \end{array}
                            \!\! \right)  \\
&=& \frac{1}{(k!)^2} \, \sum_{i_1<i_2<\cdots<i_k} {\det}^2
\left[
  \begin{array}{ccc}
    \{x_2-x_1\}_{i_1} & \cdots & \{x_{k+1}-x_1\}_{i_1} \\
    \vdots & \vdots & \vdots \\
    \{x_2-x_1\}_{i_k} & \cdots & \{x_{k+1}-x_1\}_{i_k} \\
  \end{array}
\right]  \\
&=& \frac{1}{(k!)^2} \, \sum_{i_1<i_2<\cdots<i_k} {\det}^2
\left[
  \begin{array}{ccc}
    \{x_1\}_{i_1} & \cdots & \{x_{k+1}\}_{i_1} \\
    \vdots & \vdots & \vdots \\
    \{x_1\}_{i_k} & \cdots & \{x_{k+1}\}_{i_k} \\
    1 & \cdots & 1
  \end{array}
\right] \,,
\end{eqnarray*}
where $\{x\}_i$ denotes the $i$-th component of vector $x$. Also, for all $i_1<i_2<\cdots<i_k$,
$$
{\det}^2
\left[
  \begin{array}{ccc}
    \{x_1\}_{i_1} & \cdots & \{x_{k+1}\}_{i_1} \\
    \vdots & \vdots & \vdots \\
    \{x_1\}_{i_k} & \cdots & \{x_{k+1}\}_{i_k} \\
    1 & \cdots & 1
  \end{array}
\right]  = \det\left( \sum_{j=1}^{k+1} z_j z_j\TT \right)
$$
where we have denoted by $z_j$ the $k+1$-dimensional vector with components $\{x_j\}_{i_\ell}, \ell=1,\ldots,k$, and 1.
When the $x_i$ are i.i.d.\ with the probability measure $\mu$, using Lemma~\ref{L:1} we obtain \eqref{Tha}, \eqref{Thb}.
Therefore
$$
    \psi_k(\mu) = \Psi_k[V_\mu] = \frac{k+1}{k!}\, \mathcal{E}_k\{\ml_{1}[V_\mu],\ldots,\ml_{d}[V_\mu]\} \,,
$$
with $\mathcal{E}_k\{\ml_{1}[V_\mu],\ldots,\ml_{d}[V_\mu]\}$ the elementary symmetric function of degree $k$ of the $d$ eigenvalues of $V_\mu$, see, e.g., \cite[p.~10]{MarcusM64}.
Note that
$$
\mathcal{E}_k[V_\mu]=\mathcal{E}_k\{\ml_{1}[V_\mu],\ldots,\ml_{d}[V_\mu]\}=(-1)^k a_{d-k}\,,
$$
with $a_{d-k}$ the coefficient of the monomial of degree $d-k$ of the characteristic polynomial of $V_\mu$; see, e.g., \cite[p.~21]{MarcusM64}. We have in particular $\mathcal{E}_1[V_\mu]=\tr[V_\mu]$ and $\mathcal{E}_d[V\mu)]=\det[V_\mu]$.
The shift-invariance and homogeneity of degree $2$ of $\psi_k^{1/k}(\cdot)$ follow from the shift-invariance and positive homogeneity of $\SV_k$ with degree $k$. Concavity of $\Psi_k^{1/k}(\cdot)$ follows from  \cite[p.~116]{MarcusM64} (take $p=k$ in eq.~(10), with $\mathcal{E}_0=1$).
From \cite{Lopez-FR-D98-MODA}, the $\Psi_k^{1/k}(\cdot)$ are also Loewner-increasing, so that from Lemma~\ref{L:2}, for any $\mu_1$, $\mu_2$ in $\SM$ and any $\ma\in(0,1)$,
\begin{eqnarray*}
\psi_k^{1/k}[(1-\ma)\mu_1+\ma\mu_2] &=&  \Psi_k^{1/k}\{V_{(1-\ma)\mu_1+\ma\mu_2}\}\\
                                    && \geq \Psi_k^{1/k}[(1-\ma)V_{\mu_1}+\ma V_{\mu_2}] \\
                                    && \geq (1-\ma)\Psi_k^{1/k}[V_{\mu_1}]+\ma \Psi_k^{1/k}[V_{\mu_2}] \\
                                    && = (1-\ma)\psi_k^{1/k}(\mu_1)+\ma \psi_k^{1/k}(\mu_2)\,. \hspace{2.5cm} \mbox{\fin}
\end{eqnarray*}

\vsp
The functionals $\mu\longrightarrow \phi_{[k],2,\tau}(\mu)=\psi_k^\tau(\mu)$ are thus concave for $0<\tau\leq 1/k$, with $\tau=1/k$ yielding positive homogeneity of degree 2. The functional $\psi_1(\cdot)$ is a quadratic entropy $Q_R$, see \eqref{quadratic-entropy}, or diversity measure \citep{rao2010quadratic}; $\psi_d(\mu)$ is proportional to Wilks generalised variance. Functionals $\psi_2^{1/2}(\cdot)$, see \eqref{psi2}, and more generally $\psi_k^{1/k}(\cdot)$ for $k\geq 2$, can also be considered as diversity measures.

From the well-known expression of the coefficients of the characteristic polynomial of a matrix $V$, we have
\begin{eqnarray}\label{Psi_kE_k}
  \lefteqn{ \hspace{1cm} \Psi_k(V) = \frac{k+1}{k!} \mathcal{E}_k(V) } \\
  && \hspace{1cm} = \frac{k+1}{(k!)^2} \, \det\left[
                                                 \begin{array}{cccc}
                                                   \tr(V) & k-1 & 0 & \cdots \\
                                                   \tr(V^2) & \tr(V) & k-2 & \cdots \\
                                                   \cdots & \cdots & \cdots & \cdots \\
                                                   \tr(V^{k-1}) & \tr(V^{k-2}) & \cdots & 1 \\
                                                   \tr(V^k) & \tr(V^{k-1}) & \cdots & \tr(V) \\
                                                 \end{array}
                                               \right]\,, \nonumber
\end{eqnarray}
see, e.g., \cite[p.~28]{Macdonald95}, and the $\mathcal{E}_k(V)$ satisfy the recurrence relations (Newton identities):
\begin{equation}\label{Newton}
  \mathcal{E}_k(V)= \frac1k\, \sum_{i=1}^k (-1)^{i-1}\, \mathcal{E}_{k-i}(V)\,\mathcal{E}_1(V^i) \,,
\end{equation}
see, e.g., \cite[Vol.~1, p.~88]{Gantmacher66} and \cite{Lopez-FR-D98-MODA}.
%
%
Particular forms of $\psi_k(\cdot)$ are
\begin{eqnarray*}
k=1: && \psi_1(\mu) = 2\, \tr(V_\mu)\,,  \\
k=2: && \psi_2(\mu) = \frac34\, [\tr^2(V_\mu) - \tr(V_\mu^2)] \,,  \\
k=3: && \psi_3(\mu) = \frac19\, [\tr^3(V_\mu) - 3\,\tr(V_\mu^2)\tr(V_\mu)+2\,\tr(V_\mu^3)] \,, \\
k=d: && \psi_d(\mu) = \frac{d+1}{d!}\, \det(V_\mu) \,.
\end{eqnarray*}

\subsection{Other concave homogeneous functionals}

From the proof of Theorem~\ref{Prop:1}, any Loewner-increasing, concave and homogeneous functional of the covariance matrix $V_\mu$ satisfies all properties (a)-(c) of Section~\ref{S:intro}. In particular, consider Kiefer's $\Phi_p$-class \citep{Kiefer74}, defined by
\begin{equation}\label{Kiefer-phi_p}
  \varphi_p(\mu) = \Phi_p(V_\mu) = \left\{ \begin{array}{ll}
\ml_{\max}(V_\mu) & \mbox{ for } p=\infty \,, \\
\{\frac1d\, \tr(V^p_\mu)\}^{1/p} & \mbox{ for } p \neq 0, \pm\infty \,, \\
\det^{1/d}(V_\mu) & \mbox{ for } p = 0 \,, \\
\ml_{\min}(V_\mu) & \mbox{ for } p=-\infty \,,\\
\end{array} \right.
\end{equation}
with the continuous extension $\varphi_p(\mu)=0$ for $p<0$ when $V_\mu$ is singular.
Notice that $\varphi_1(\cdot)$ and $\varphi_0(\cdot)$ respectively coincide with $\psi_1(\cdot)$ and $\psi_d^{1/d}(\cdot)$ (up to a multiplicative scalar).

The functionals $\varphi_p(\cdot)$ are homogeneous of degree 2, and concave for $p\in[-\infty,1]$, see, e.g., \cite[Chap.~6]{Pukelsheim93}. However, by construction, for any $p \leq 0$, $\varphi_p(\mu)=0$ when $\mu$ is concentrated in a $q$-dimensional subspace of $\mathds{R}^d$, for any $q<d$, whereas $\varphi_p(\mu)>0$ for $p>0$ and any $q>0$. The family of functionals \eqref{Kiefer-phi_p} is therefore unable to detect the true dimensionality of the data. On the other hand, $\psi_k(\mu)=0$ for all $k>q$ when rank $V_\mu=q$.

\subsection{Empirical version and unbiased estimates}\label{S:empirical}

Let $x_1,\ldots,x_n$ be a sample of $n$ vectors of $\mathds{R}^d$, i.i.d.\ with the measure $\mu$. This sample can be used to obtain an empirical estimate $({\widehat\psi}_1)_n$ of $\psi_k(\mu)$, through the consideration of the ${n \choose k+1}$ $k$-dimensional simplices that can be constructed with the $x_i$. Below we show how a much simpler (and still unbiased) estimation of $\psi_k(\mu)$ can be obtained through the empirical variance-covariance matrix of the sample. See also \cite{wilks1960, Wilks1962}.

Denote
\begin{eqnarray*}
\widehat x_n &=& \frac1n\, \sum_{i=1}^n x_i \,, \\
\widehat V_n &=& \frac{1}{n-1}\, \sum_{i=1}^n (x_i-\widehat x_n)(x_i-\widehat x_n)\TT = \frac{1}{n(n-1)}\, \sum_{i<j} (x_i-x_j)(x_i-x_j)\TT \,,
\end{eqnarray*}
respectively the empirical mean and variance-covariance matrix of $x_1$. Note that both are unbiased. 
For all $k\in\{1,\ldots,d\}$ we define $\psi_k(\mu_n)=\Psi_k(\widehat V_n)$. 
We thus have
$$
({\widehat\psi}_1)_n  = \frac{2}{n(n-1)}\, \sum_{i<j} \|x_i-x_j\|^2= 2\, \tr[\widehat V_n]=\Psi_1(\widehat V_n)=\psi_1(\mu_n)\,,
$$
with $\mu_n$ the empirical measure of the sample, and the estimator $({\widehat\psi}_1)_n$ is an unbiased estimator of $\psi_1(\mu)$.
%
For $k\geq 1$, consider the empirical estimate
\begin{equation}\label{widehat-psi}
({\widehat\psi}_k)_n  = {n \choose k+1}^{-1}\, \sum_{j_1<j_2<\cdots<j_{k+1}} \SV_k^2(x_{j_1},\ldots,x_{j_{k+1}}) \,.
\end{equation}
Is satisfies the following.

\begin{theorem}\label{Th:empirical} For $x_1,\ldots,x_n$ a sample of $n$ vectors of $\mathds{R}^d$, i.i.d.\ with the measure $\mu$, and for any $k\in\{1,\ldots,d\}$, we have
\begin{equation}
({\widehat\psi}_k)_n  = \frac{(n-k-1)!(n-1)^k}{(n-1)!}\, \Psi_k(\widehat V_n)=\frac{(n-k-1)!(n-1)^k}{(n-1)!}\, \psi_k(\mu_n)\,, \label{widehat-psi-b}
\end{equation}
and $({\widehat\psi}_k)_n $ forms an unbiased estimator of $\psi_k(\mu)$ with minimum variance among all unbiased estimators.
\end{theorem}

This result generalises the main result of \citet{van1965note} to $k\leq d$, see Corollary~2.1 in that paper. The proof is given in Appendix.

Using the notation of Theorem~\ref{Prop:1}, since $\mathcal{E}_k(V)=(-1)^k a_{d-k}(V)$, with $a_{d-k}(V)$ the coefficient of the monomial of degree $d-k$ of the characteristic polynomial of $V$, for a nonsingular $V$ we obtain
\begin{equation}\label{EkEd-k}
  \mathcal{E}_k(V)=\det(V)\,\mathcal{E}_{d-k}(V^{-1}) \,,
\end{equation}
see also \cite[Eq.~4.2]{Lopez-FR-D98-MODA}. Therefore, we also have
\begin{equation}
({\widehat\psi}_{d-k})_n  = \frac{(n-d+k-1)!(n-1)^{d-k}}{(n-1)!}\,\frac{(d-k+1)k!}{(k+1)(d-k)!} \det(\widehat V_n)\,\Psi_k(\widehat V_n^{-1}) \,, \label{widehat-psi-c}
\end{equation}
which forms an unbiased and minimum-variance estimator of $\psi_{d-k}(\mu)$.
Note that the estimation of $\psi_k(\mu)$ is much simpler through \eqref{widehat-psi-b} or \eqref{widehat-psi-c} than using the direct construction \eqref{widehat-psi}.

One may notice that $\psi_1(\mu_n)$ is clearly unbiased due to the linearity of $\Psi_1(\cdot)$, but it is remarkable that $\psi_k(\mu_n)$ becomes unbiased after a suitable scaling, see \eqref{widehat-psi-b}.
Since $\Psi_k(\cdot)$ is highly nonlinear for $k>1$, this property would not hold if $\widehat V_n$ were replaced by another unbiased estimator of $V_\mu$.

\vsp
The value of $({\widehat\psi}_k)_n$ only depend on $\widehat V_n$, with $\Ex\{({\widehat\psi}_k)_n\}=\psi_k(V_\mu)$, but its variance depends on the distribution itself.
Assume $\Ex\{\SV_k^4(x_1,\ldots,x_{k+1})\}<\infty$.
From \cite[Lemma A, p.~183]{Serfling80}, the variance of $({\widehat\psi}_k)_n$ satisfies
$$
\var[({\widehat\psi}_k)_n]=\frac{(k+1)^2}{n}\, \omega + O(n^{-2}) \,,
$$
where $\omega= \var[h(x)]$, with $h(x)=\Ex\{\SV_k^2(x_1,x_2,\ldots,x_{k+1})|x_1=x\}$. Obviously, $\Ex[h(x)]=\psi_k(\mu)$ and calculations similar to those in the proof of Theorem~\ref{Prop:1} give
\begin{eqnarray}
&& \hspace{0.5cm} \omega = \frac{1}{(k!)^2}\, \sum_{I,J}  \det[\{V_\mu\}_{I\times I}] \, \det[\{V_\mu\}_{J\times J}]  \label{zeta} \\
&& \times\, \left[\Ex\left\{ (E_\mu-x)_I\TT \{V_\mu\}_{I\times I}^{-1} (E_\mu-x)_I (E_\mu-x)_J\TT \{V_\mu\}_{J\times J}^{-1} (E_\mu-x)_J \right\} - k^2\right]\,, \nonumber
\end{eqnarray}
where $I$ and $J$ respectively denote two sets of indices $i_1<i_2<\cdots i_k$ and $j_1<j_2<\cdots <j_k$ in $\{1,\ldots,d\}$, the summation being over all possible such sets. Simplifications occur in some particular cases. For instance, when $\mu$ is a normal measure, then
\begin{eqnarray*}
\omega &=& \frac{2}{(k!)^2}\, \sum_{I,J}  \det[\{V_\mu\}_{I\times I}] \, \det[\{V_\mu\}_{J\times J}]\\
 && \times\, \tr\left[\{V_\mu\}_{J\times J}^{-1} \{V_\mu\}_{J\times I}\{V_\mu\}_{I\times I}^{-1}\{V_\mu\}_{I\times J} \right] \,.
\end{eqnarray*}
If, moreover, $V_\mu$ is the diagonal matrix $\diag\{\ml_1,\ldots,\ml_d\}$, then
$$
\omega = \frac{2}{(k!)^2}\, \sum_{I,J}  \beta(I,J)\, \prod_I \ml_i \prod_J \ml_j \,,
$$
with $\beta(I,J)$ denoting the number of coincident indices between $I$ and $J$ (i.e., the size of $I\cap J$).
When $\mu$ is such that the components of $x$ are i.d.d.\ with variance $\ms^2$, then $V_\mu=\ms^2 I_d$, with $I_d$ the $d$-dimensional identity matrix, and
\begin{eqnarray*}
&& \Ex\left\{ (E_\mu-x)_I\TT \{V_\mu\}_{I\times I}^{-1} (E_\mu-x)_I (E_\mu-x)_J\TT \{V_\mu\}_{J\times J}^{-1} (E_\mu-x)_J\right\} = \\
&& \hspace{3cm} \Ex\left\{ \left(\sum_{i\in I} z_i^2\right)\left(\sum_{j\in J} z_j^2\right) \right\}\,,
\end{eqnarray*}
where the $z_i=\{x-E_\mu\}_i/\ms$ are i.i.d.\ with mean 0 and variance 1. We then obtain
$$
\omega = \frac{\ms^{4k}}{(k!)^2}\, (\Ex\{z_i^4\}-1)\, \beta_{d,k} \,,
$$
where
\begin{eqnarray*}
\beta_{d,k}=\sum_{I,J}  \beta(I,J) &=& \sum_{i=1}^k i\, {d \choose i}\, {d-i \choose k-i}\, {d-i-(k-i) \choose k-i} \\
&=& \frac{(d-k+1)^2}{d}\, {d \choose k-1}^2 \,.
\end{eqnarray*}

\paragraph{Example 1}

We generate $1,000$ independent samples of $n$ points for different measures $\mu$. Figure~\ref{F:boxplot_unif_1e2_1000repet} presents a box-plot of the ratios $({\widehat\psi}_k)_n/\psi_k(\mu)$ for various values of $k$ and $n=100$ (left), $n=1,000$ (right), when $\mu=\mu_1$ uniform in $[0,1]^{10}$. Figure~\ref{F:boxplot_normal_1e2_1000repet} presents the same information when $\mu=\mu_2$ which corresponds to the normal distribution $\SN(0,I_{10}/12)$ in $\mathds{R}^{10}$. Note that $V_{\mu_1}=V_{\mu_2}$ but the dispersions are different in the two figures. The fact that the variance of the ratio $({\widehat\psi}_k)_n/\psi_k(\mu)$ increases with $k$ is due to the decrease of $\psi_k(\mu)$, see Figure~\ref{F:means-ratio_normal_n100_nrepet1000}-left. Note that the values of $\psi_k(\mu)$ and empirical mean of $({\widehat\psi}_k)_n$ are extremely close.  Figure~\ref{F:means-ratio_normal_n100_nrepet1000}-right presents the asymptotic and empirical variances of $({\widehat\psi}_k)_n/\psi_k(\mu)$ as functions of $k$.


\begin{figure}[ht]
\centering
\includegraphics[width=0.37\textwidth, angle=90]{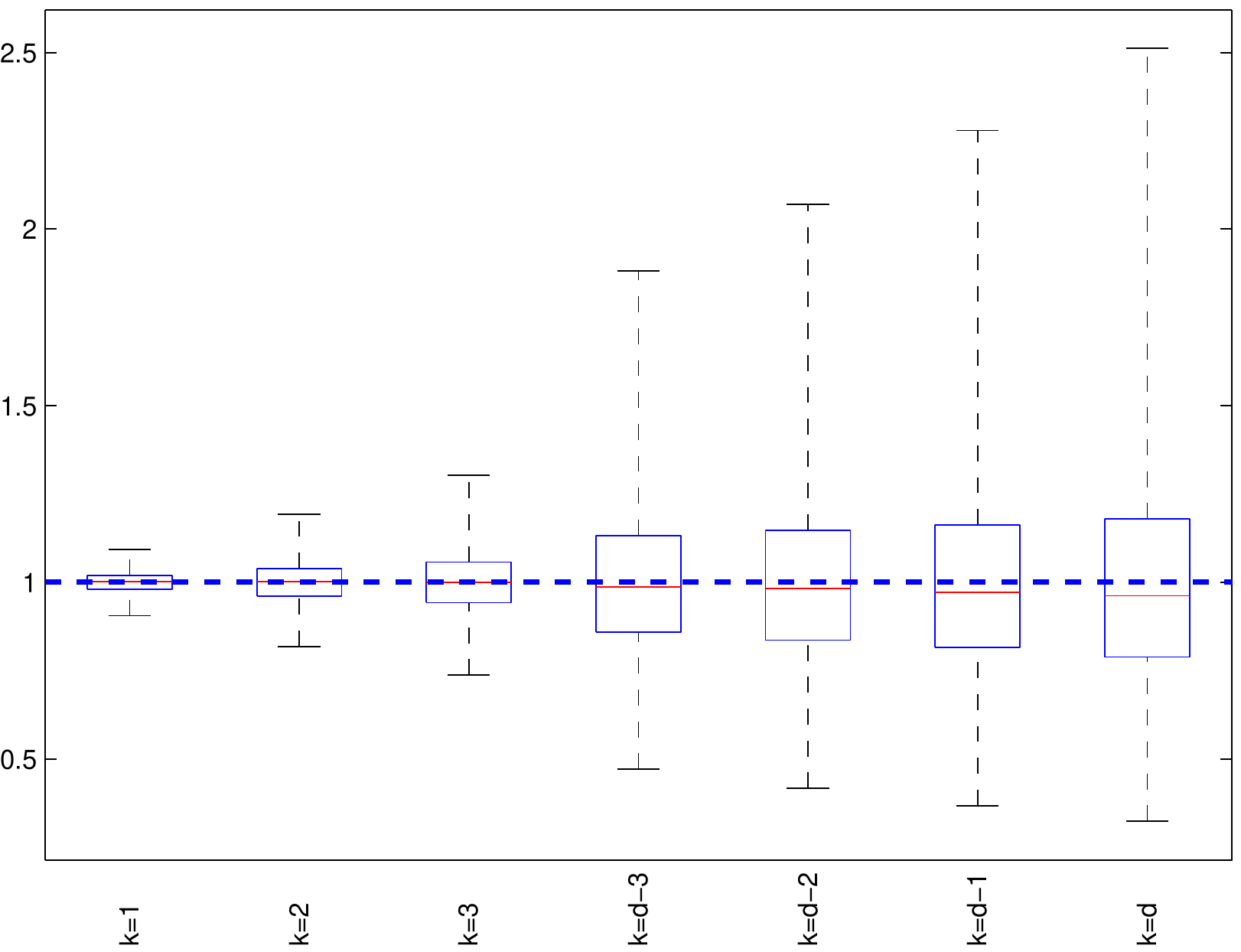}
\includegraphics[width=0.37\textwidth, angle=90]{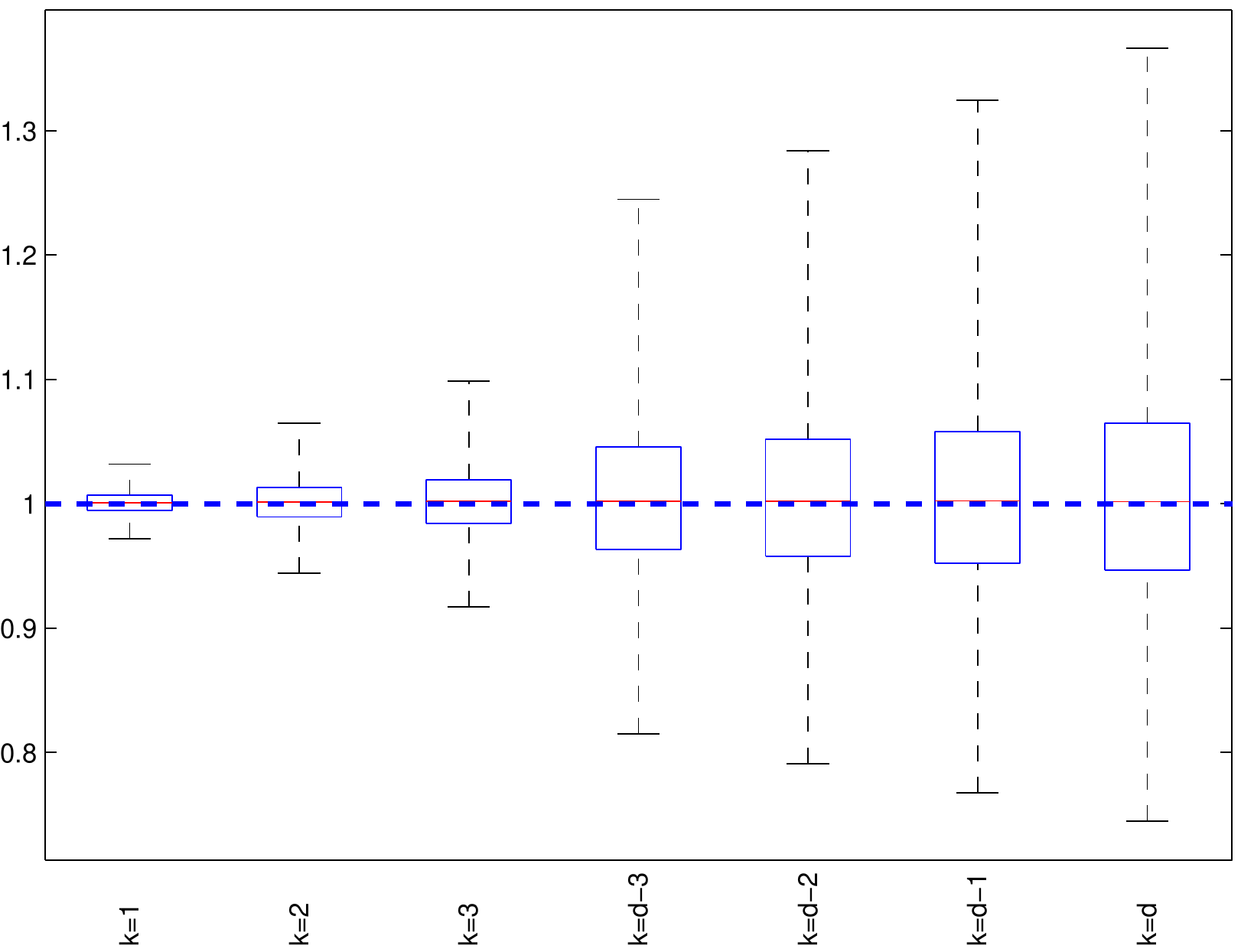}
\caption{Box-plot of $({\widehat\psi}_k)_n/\psi_k(\mu)$ for different values of $k$: $\mu$ is uniform in $[0,1]^{10}$, $n=100$ (Left) and $n=1,000$ (Right) --- 1,000 repetitions; minimum, median and maximum values are indicated, together with 25\% and 75\% quantiles.} \label{F:boxplot_unif_1e2_1000repet}
\end{figure}

\begin{figure}[ht]
\centering
\includegraphics[width=0.37\textwidth, angle=90]{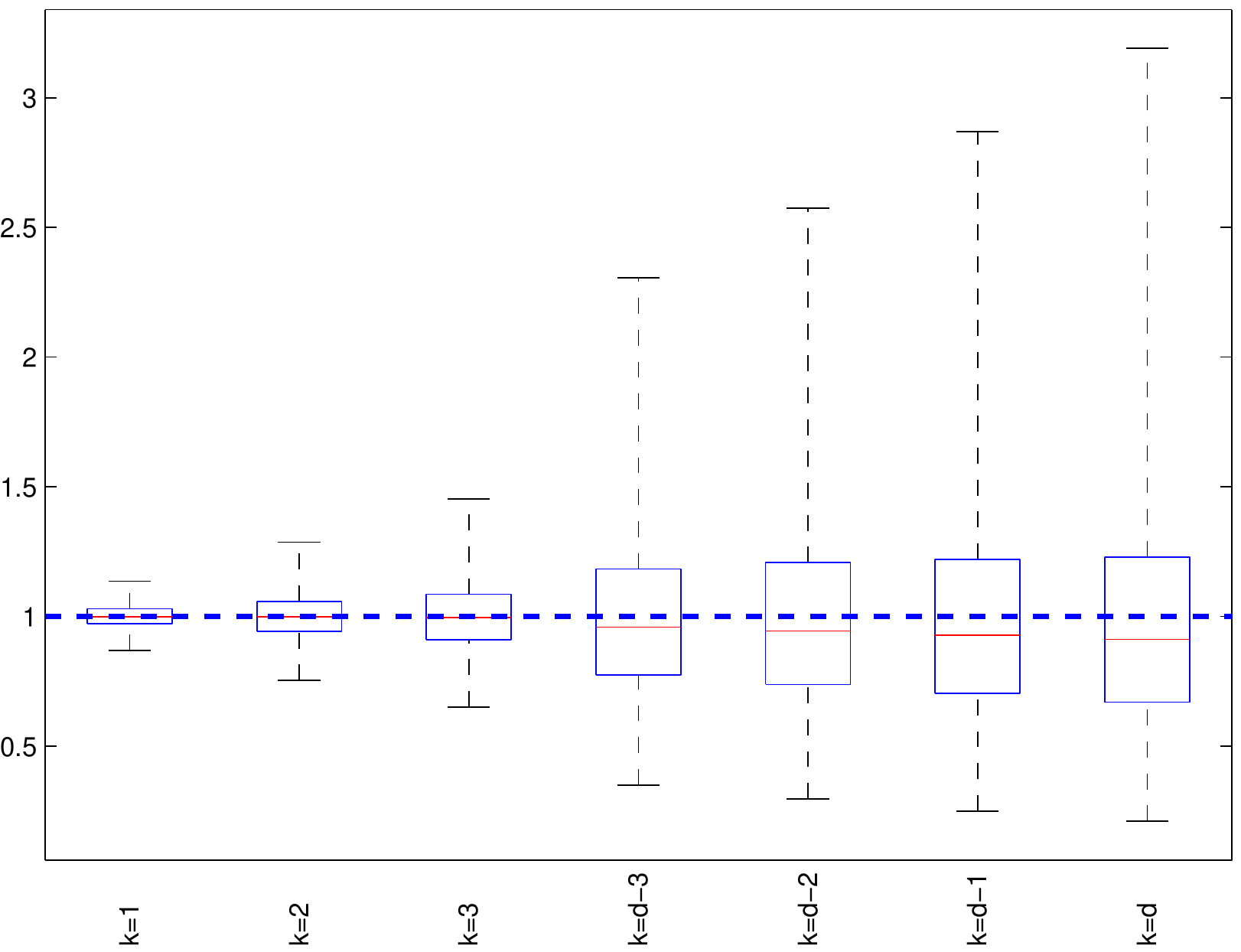}
\includegraphics[width=0.37\textwidth, angle=90]{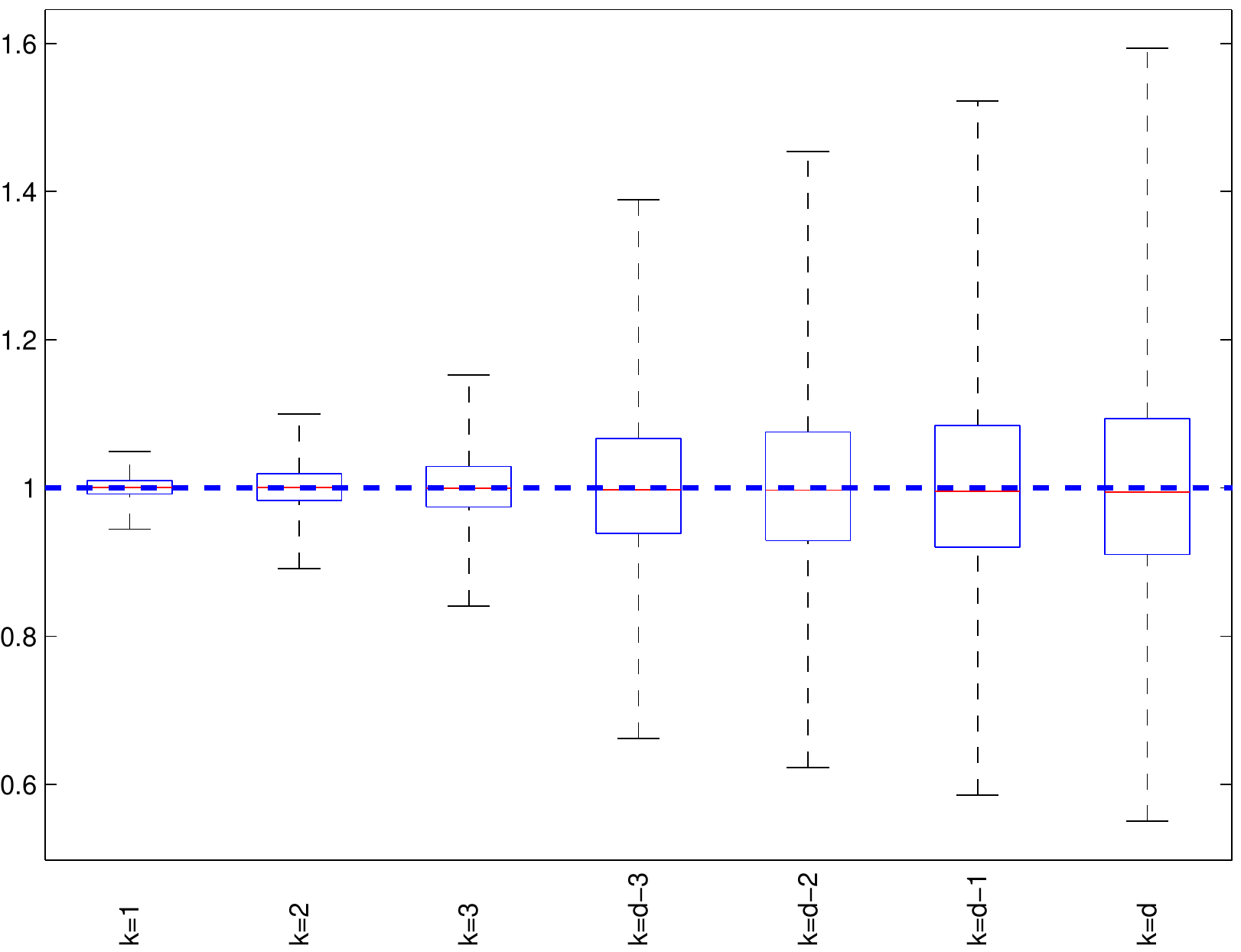}
\caption{Same as in Figure~\ref{F:boxplot_unif_1e2_1000repet} but for $\mu$ normal $\SN(0,I_{10}/12)$.} \label{F:boxplot_normal_1e2_1000repet}
\end{figure}

\begin{figure}[ht]
\centering
\includegraphics[width=0.47\textwidth]{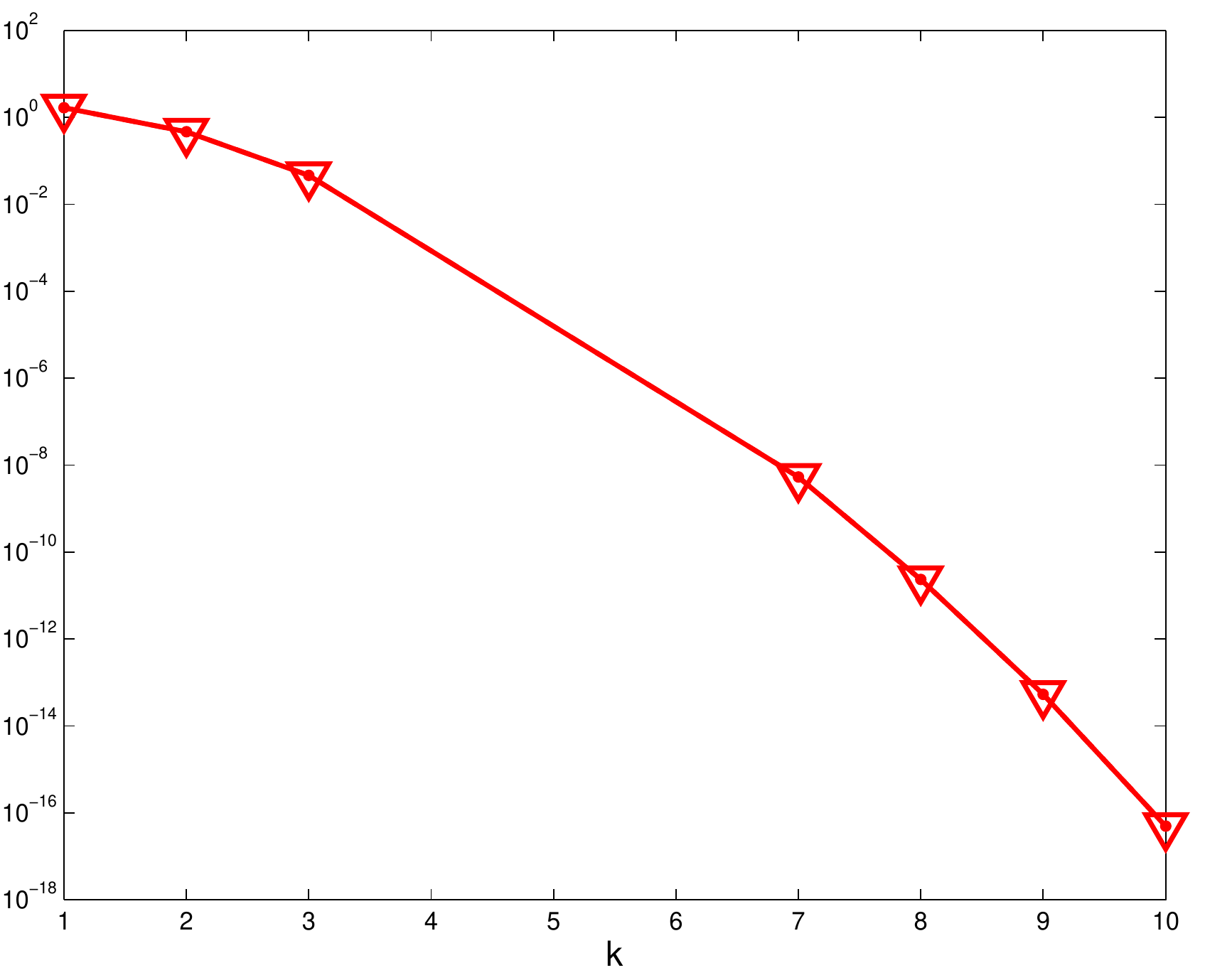}
\includegraphics[width=0.47\textwidth]{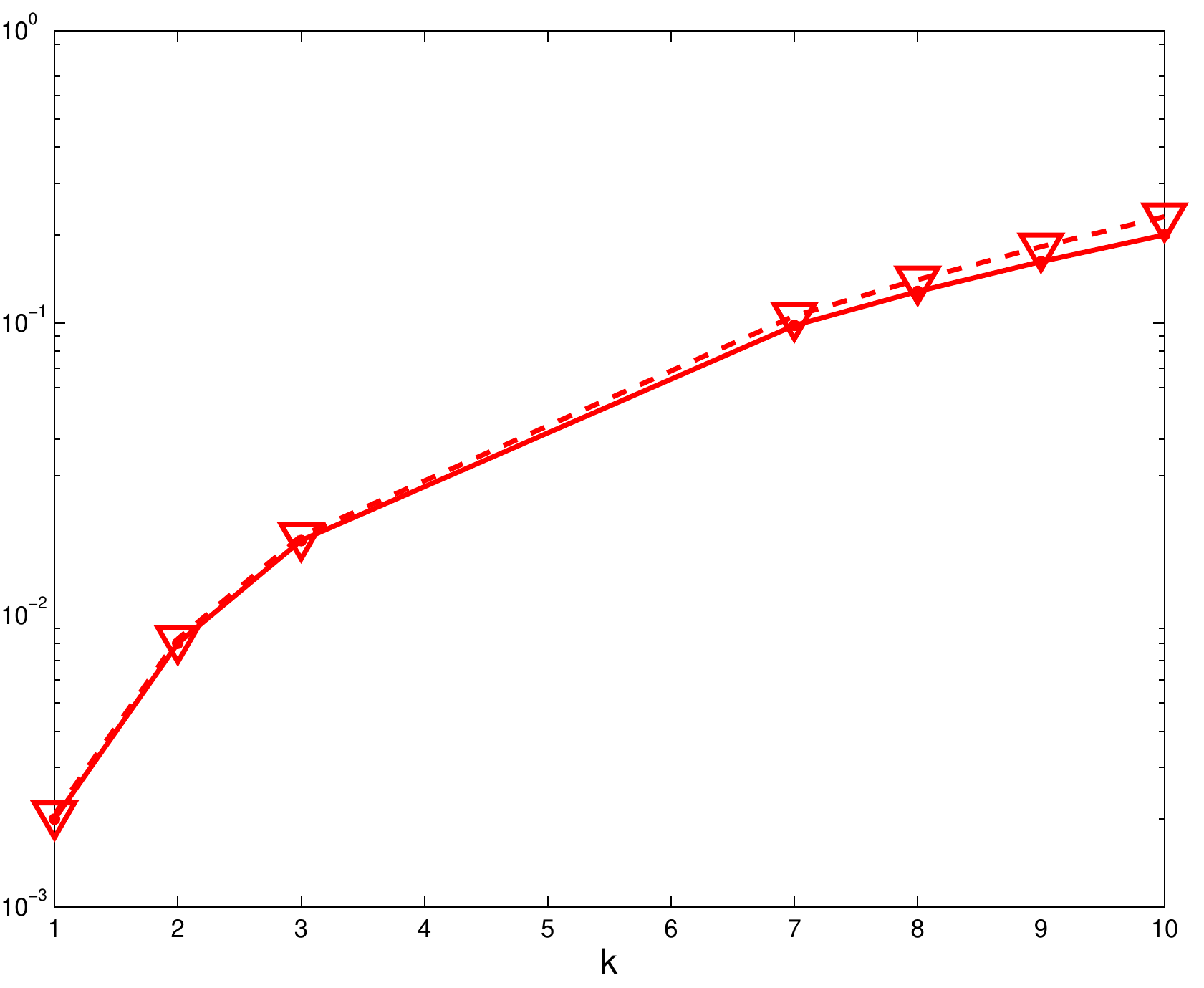}
\caption{Left: $\psi_k(\mu)$ (dots and solid line) and empirical mean of $({\widehat\psi}_k)_n$ (triangles and dashed line); Right:
asymptotic (dots and solid line) and empirical (triangles and dashed line) variances of $({\widehat\psi}_k)_n/\psi_k(\mu)$;
$\mu$ is normal $\SN(0,I_{10}/12)$, $n=100$, 1,000 repetitions.} \label{F:means-ratio_normal_n100_nrepet1000}
\end{figure}


Other properties of U-statistics apply to the estimator $({\widehat\psi}_k)_n$, including almost-sure consistency and the classical law of the iterated logarithm, see \cite[Section~5.4]{Serfling80}. In particular, $({\widehat\psi}_k)_n$ is asymptotically normal $\SN(\psi_k(\mu),(k+1)^2\omega/n)$, with $\omega$ given by \eqref{zeta}. This is illustrated in Figure~\ref{F:normal}-left below for $\mu$ uniform in $[0,1]^{10}$, with $n=1,000$ and $k=3$. The distribution is already reasonably close to normality for small values of $n$, see Figure~\ref{F:normal}-right for which $n=20$.

\begin{figure}[ht]
\centering
\includegraphics[width=0.47\textwidth]{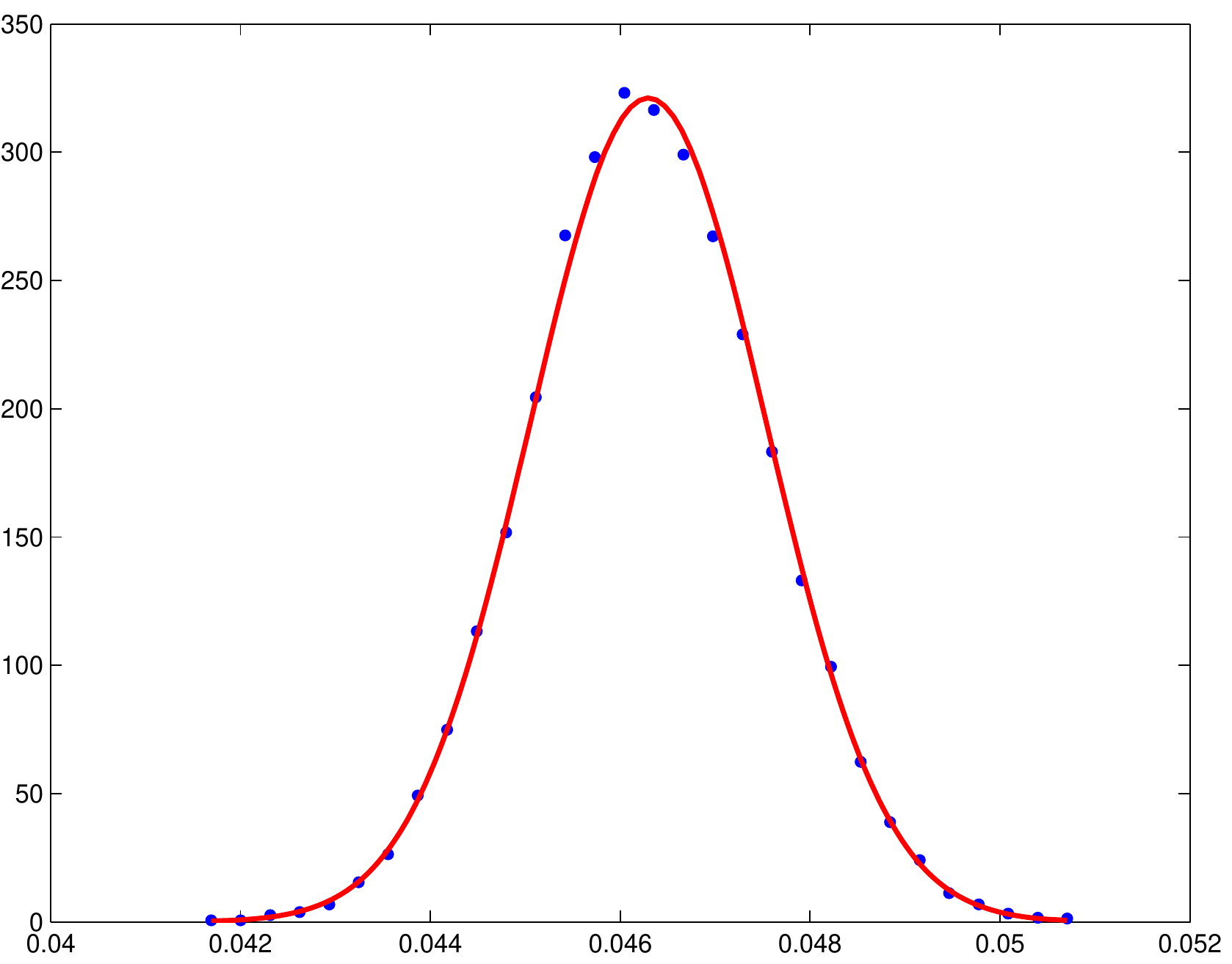}
\includegraphics[width=0.47\textwidth]{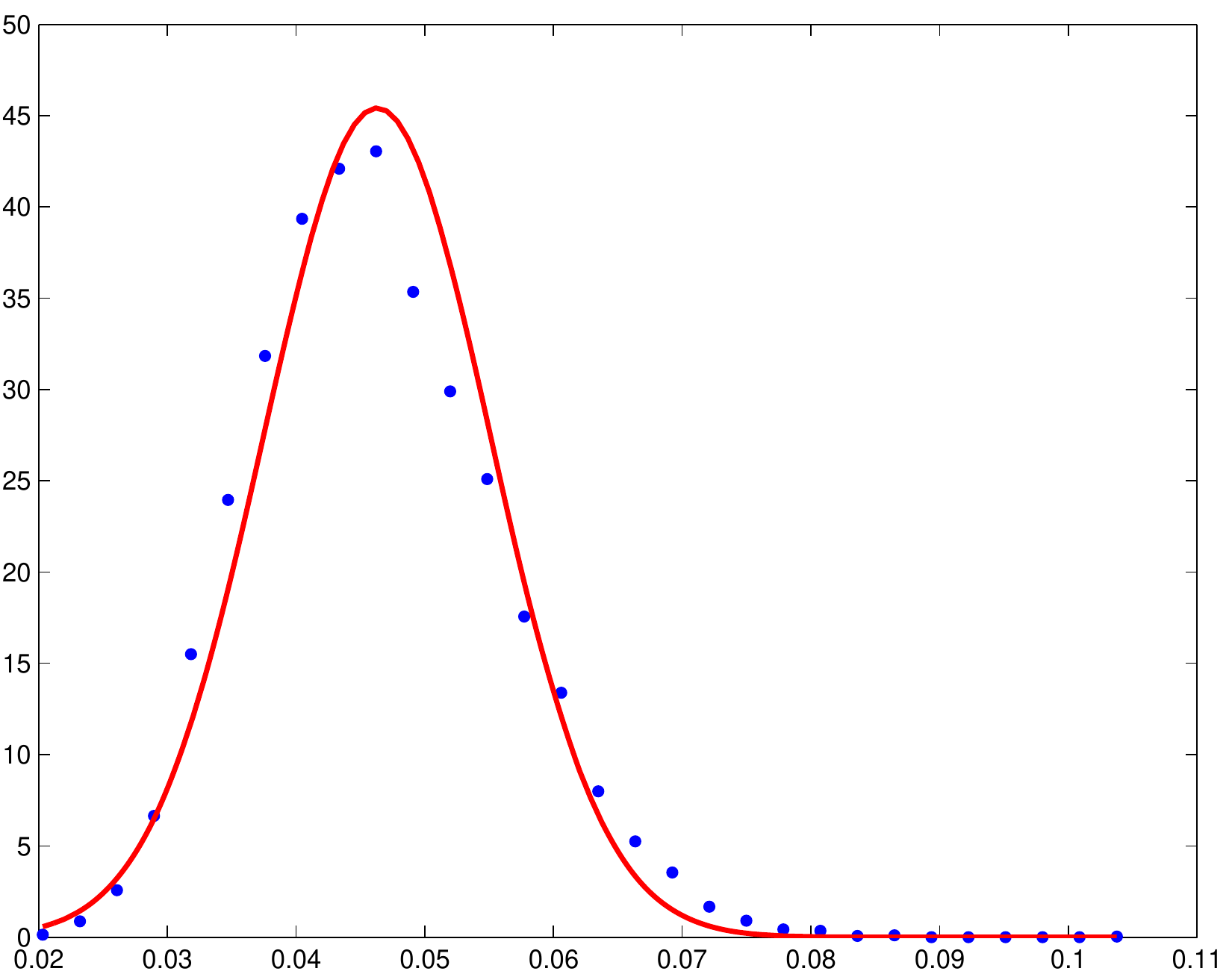}
\caption{Dots: empirical distribution of $({\widehat\psi}_k)_n$ (histogram for 10,000 independent repetitions); solid line: asymptotic normal distribution $\SN(\psi_k(\mu),(k+1)^2\omega/n)$; $\mu$ is uniform in $[0,1]^{10}$ and $k=3$; left: $n=1,000$; right: $n=20$.} \label{F:normal}
\end{figure}


\section{Maximum-diversity measures and optimal designs}\label{S:maxent-and-design}

In this section we consider two types of optimisation problems on $\SM$ related to the functionals $\psi_k(\cdot)$ introduced in Theorem~\ref{Prop:1}. First, in Section~\ref{S:maxent}, we are interested in the characterisation and construction of maximum-diversity measures; that is, measures $\mu_k^*\in\SM$ which maximize $\psi_k(\mu)=\Psi_k(V_\mu)$. The existence of an optimal measure follows from the compactness of $\SX$ and continuity of $\SV_k(x_1,\ldots,x_{k+1})$ in each $x_i$, see \cite[Th.~1]{Bjorck56}; the concavity and differentiability of the functional $\psi_k^{1/k}(\cdot)$ allow us to derive a necessary and sufficient condition for optimality.

In Section~\ref{S:design} we consider the problem of optimal design of experiments, where the covariance matrix $V$ is the inverse of the information matrix $M(\xi)$ for some regression model.

\subsection{Maximum-diversity measures}\label{S:maxent}

\subsubsection{Necessary and sufficient condition}
Since the functionals $\psi_k^{1/k}(\cdot)$ are concave and differentiable, for all $k=1,\ldots,d$, we can easily derive a necessary and sufficient condition for a probability measure $\mu_k^*$ on $\SX$ to maximise $\psi_k(\mu)$, in the spirit of the celebrated Equivalence Theorem of \cite{KieferW60}.

Denote by $\nabla_{\Psi_k}[V]$ the gradient of $\Psi_k(\cdot)$ at matrix $V$ (a matrix of the same size as $V$) and by $F_{\psi_k}(\mu;\nu)$ the directional derivative of $\psi_k(\cdot)$ at $\mu$ in the direction $\nu$;
$$
F_{\psi_k}(\mu;\nu) = \lim_{\ma\ra 0^+} \frac{\psi_k[(1-\ma)\mu+\ma\nu]-\psi_k(\mu)}{\ma} \,.
$$
From the expression \eqref{Psi_kE_k} of $\Psi_k(V)$, we have
$$
\nabla_{\Psi_k}[V]=\frac{k+1}{k!} \nabla_{\mathcal{E}_k}[V]\,,
$$
where $\nabla_{\mathcal{E}_k}[V]$ denotes the gradient of $\mathcal{E}_k(\cdot)$ at $V$, which, using \eqref{Newton}, can be shown by induction to satisfy
\begin{equation}\label{nabla-E_k}
  \nabla_{\mathcal{E}_k}[V] = \sum_{i=0}^{k-1} (-1)^i\, \mathcal{E}_{k-i-1}(V)\,V^i \,,
\end{equation}
see \cite{Lopez-FR-D98-MODA}. We thus obtain in particular
\begin{eqnarray*}
k=1: && \nabla_{\Psi_1}[V]= 2\, I_d \,, \\
k=2: && \nabla_{\Psi_2}[V]= \frac32\, [\tr(V)\,I_d - V] \,, \\
k=3: && \nabla_{\Psi_3}[V]= \frac13\, [\tr^2(V) - \tr(V^2)]\,I_d - \frac23\, \tr(V)\,V + \frac23\, V^2 \,,\\
k=d: && \nabla_{\Psi_d}[V]= \frac{d+1}{d!}\, \det(V)\,V^{-1} \,.
\end{eqnarray*}

Using the differentiability of $\Psi_k(\cdot)$, direct calculation gives
$$
F_{\psi_k}(\mu;\nu) =  \tr\left\{\nabla_{\Psi_k}[V_\mu]\,\frac{\dd V_{(1-\ma)\mu+\ma\nu}}{\dd\ma}\bigg|_{\ma=0} \right\}\,,
$$
with
\begin{equation}\label{derivativeV}
\frac{\dd V_{(1-\ma)\mu+\ma\nu}}{\dd\ma}\bigg|_{\ma=0} = \int [xx\TT -(E_{\mu} x\TT+x E_{\mu}\TT)]\,\nu(\dd x) - \int xx\TT\,\mu(\dd x) + 2 E_{\mu} E_{\mu}\TT \,.
\end{equation}
Notice that $\dd V_{(1-\ma)\mu+\ma\nu}/\dd\ma\big|_{\ma=0}$ is linear in $\nu$.

Then, from the concavity of $\psi_k^{1/k}(\cdot)$, $\mu_k^*$ maximises $\psi_k(\mu)$ with respect to $\mu\in\SM$ if and only if $\psi_k(\mu_k^*)>0$ and $F_{\psi_k}(\mu_k^*;\nu) \leq 0$ for all $\nu\in\SM$, that is
\begin{equation}\label{CNS-nu}
\tr\left\{\nabla_{\Psi_k}[V_{\mu_k^*}]\,\frac{\dd V_{(1-\ma)\mu_k^*+\ma\nu}}{\dd\ma}\bigg|_{\ma=0} \right\} \leq 0 \,, \ \forall\nu\in\SM\,.
\end{equation}
We obtain the following.

\begin{theorem}\label{th:equivTh} The probability measure $\mu_k^*$ such that $\psi_k(\mu_k^*)>0$ is $\psi_k$-optimal, that is, maximises $\psi_k(\mu)$ with respect to $\mu\in\SM$, $k\in\{1,\ldots,d\}$, if and only if
\begin{equation}\label{CNS}
  \max_{x\in\SX} (x-E_{\mu_k^*})\TT \frac{\nabla_{\Psi_k}[V_{\mu_k^*}]}{\Psi_k(V_{\mu_k^*})}(x-E_{\mu_k^*}) \leq k \,.
\end{equation}
Moreover,
\begin{equation}\label{support-mu*}
(x-E_{\mu_k^*})\TT \frac{\nabla_{\Psi_k}[V_{\mu_k^*}]}{\Psi_k(V_{\mu_k^*})}(x-E_{\mu_k^*}) = k
\end{equation}
for all $x$ in the support of $\mu_k^*$.
\end{theorem}


\begin{proof}
First note that the Newton equations \eqref{Newton} and the recurrence \eqref{nabla-E_k} for $\nabla_{\mathcal{E}_k}[\cdot]$ imply that $\tr(V\nabla_{\Psi_k}[V])=k\Psi_k(V)$ for all $k=1,\ldots,d$.

The condition \eqref{CNS} is sufficient. Indeed, suppose that $\mu_k^*$ such that $\psi_k(\mu_k^*)>0$ satisfies \eqref{CNS}. We obtain
$$
\int (x-E_{\mu_k^*})\TT \nabla_{\Psi_k}[V_{\mu_k^*}](x-E_{\mu_k^*})\, \nu(\dd x) \leq \tr\left\{V_{\mu_k^*}\nabla_{\Psi_k}[V_{\mu_k^*}] \right\}
$$
for any $\nu\in\SM$, which gives \eqref{CNS-nu} when we use \eqref{derivativeV}. The condition is also necessary since \eqref{CNS-nu} must be true in particular for $\delta_x$, the delta measure at any $x\in\SX$, which gives \eqref{CNS}. The property \eqref{support-mu*} on the support of $\mu_k^*$ follows from the observation that $\int (x-E_{\mu_k^*})\TT \nabla_{\Psi_k}[V_{\mu_k^*}](x-E_{\mu_k^*})\, \mu_k^*(\dd x) = \tr\left\{V_{\mu_k^*}\nabla_{\Psi_k}[V_{\mu_k^*}] \right\}$.
\end{proof}

Note that for $k<d$, the covariance matrix $V_{\mu_k^*}$ of a $\psi_k$-optimal measure $\mu_k^*$ is not necessarily unique and may be singular; see, e.g., Examples~2 and 3 in Section~\ref{S:Examples}.  Also, $\psi_k(\mu)>0$ implies that $\psi_{k-1}(\mu)>0$, $k=2,\ldots,d$.

\begin{remark} As a natural extension of the concept of potential in case of order-two interactions ($k=1$), we call
$P_{k,\mu}(x)= \psi_k(\mu,\ldots,\mu,\delta_x)$ the potential of $\mu$ at $x$, where
$$
\psi_k(\mu_1,\ldots,\mu_{k+1})= \int \ldots \int \SV_k^2(x_1,\ldots,x_{k+1})\, \mu_1(\dd x_1) \ldots \mu_{k+1}(\dd x_{k+1}) \,.
$$
This yields $F_{\psi_k}(\mu;\nu) = (k+1)\, [\psi_k(\mu,\ldots,\mu,\nu)-\psi_k(\mu)]$, where $\mu$ appears $k$ times in $\psi_k(\mu,\ldots,\mu,\nu)$. Therefore, Theorem~\ref{th:equivTh} states that $\mu_k^*$ with $\psi_k(\mu_k^*)>0$ is $\psi_k$-optimal if and only if $\psi_k(\mu_k^*,\ldots,\mu_k^*,\nu) \leq \psi_k(\mu_k^*)$ for any $\nu\in\SM$, or equivalently $P_{k,\mu_k^*}(x) \leq \psi_k(\mu_k^*)$ for all $x\in\SX$.

It can be shown that for any measure $\mu\in\SM$, $\min_{x\in\SX} P_{k,\mu}(x)$ is reached for $x=E_\mu$, which extends the result of \citet{wilks1960} about the minimum property of the internal scatter.
\end{remark}

\begin{remark} \label{R:Radoslav}
Consider Kiefer's $\Phi_p$-class of orthogonally invariant criteria and their associated functional $\varphi_p(\cdot)$, see \eqref{Kiefer-phi_p}.
From a result in \citep{Harman2004-MODA}, if a measure $\mu_p$ optimal for some $\varphi_p(\cdot)$ with $p\in(-\infty,1]$ is such that $V_{\mu_p}$ is proportional to the identity matrix $I_d$, then $\mu_p$ is simultaneously optimal for all orthogonally invariant criteria. A measure $\mu_p$ having this property is therefore $\psi_k$-optimal for all $k=1,\ldots,d$.
\end{remark}

\begin{remark}\label{R:otherET} Using \eqref{EkEd-k}, when $V$ is nonsingular we obtain the property
$$
\Psi_k(V)=\frac{(k+1)(d-k)!}{(d-k+1)k!}\, \det(V)\, \Psi_{d-k}(V^{-1})
$$
which implies that maximising $\Psi_k(V)$ is equivalent to maximising $\log\det(V)+\log\Psi_{d-k}(V^{-1})$. Therefore, Theorem~\ref{th:equivTh} implies that $\mu_k^*$ with nonsingular covariance matrix $V_{\mu_k^*}$ maximises $\psi_k(\mu)$ if and only if
$$
  \max_{x\in\SX} (x-E_{\mu_k^*})\TT \left[V_{\mu_k^*}^{-1}- V_{\mu_k^*}^{-1}\,\frac{\nabla_{\Psi_{d-k}}[V_{\mu_k^*}^{-1}]}{\Psi_{d-k}(V_{\mu_k^*}^{-1})}\,V_{\mu_k^*}^{-1} \right](x-E_{\mu_k^*}) \leq d- k \,,
$$
with equality for $x$ in the support of $\mu_k^*$. When $k$ is large (and $d-k$ is small), one may thus check the optimality of $\mu_k^*$ without using the complicated expressions of $\Psi_k(V)$ and $\nabla_{\Psi_k}[V]$.
\end{remark}

\subsubsection{A duality property}

The characterisation of maximum-diversity measures can also be approached from the point of view of duality theory.

When $k=1$, the determination of a $\psi_1$-optimal measure $\mu_1^*$ is equivalent to the dual problem of constructing the minimum-volume ball $\SB_d^*$ containing $\SX$. If this ball has radius $\rho$, then $\psi_1(\mu_1^*)=2\rho^2$, and the support points of $\mu_1^*$ are the points of contact between $\SX$ and $\SB_d^*$; see \cite[Th.~6]{Bjorck56}. Moreover, there exists an optimal measure with no more than $d+1$ points.

The determination of an optimal measure $\mu_d^*$ is also dual to a simple geometrical problem: it corresponds to the determination of the minimum-volume ellipsoid $\SE_d^*$ containing $\SX$. This is equivalent to a $D$-optimal design problem in $\mathds{R}^{d+1}$ for the estimation of $\beta=(\beta_0,\beta_1\TT)\TT$, $\beta_1\in\mathds{R}^d$, in the linear regression model with intercept $\beta_0+\beta_1\TT x$, $x\in\SX$, see \cite{Titterington75}. Indeed, denote
$$
W_\mu= \int_\SX (1\ \ x\TT)\TT (1\ \ x\TT) \, \mu(\dd x)\,.
$$
Then $\SE_{d+1}^*=\{z\in\mathds{R}^{d+1}: z\TT W^{-1}_{\mu_d^*} z \leq d+1\}$, with $\mu_d^*$ maximising $\det(W_\mu)$, is the minimum-volume ellipsoid centered at the origin and containing the set $\{z\in\mathds{R}^{d+1}: z=(1\ \ x\TT)\TT,\ x\in\SX \}$. Moreover, $\SE_d^*$ corresponds to the intersection between $\SE_{d+1}^*$ and the hyperplane $\{z\}_1=1$; see, e.g., \cite{ShorB92}. This gives $\psi_d(\mu_d^*)=(d+1)/d!\, \det(W_{\mu_d^*})$. The support points of $\mu_d^*$ are the points of contact between $\SX$ and $\SE_d^*$, there exists an optimal measure with no more than $d(d+3)/2+1$ points, see \cite{Titterington75}.

The property below generalises this duality property to any $k\in\{1,\ldots,d\}$.
\begin{theorem}\label{Th:duality}
$$
\max_{\mu\in\SM} \Psi_k^{1/k}(V_\mu) = \min_{M,c:\ \SX\subset\SE(M,c)} \frac{1}{\phi_k^\infty(M)} \,,
$$
where $\SE(M,c)$ denotes the ellipsoid $\SE(M,c) = \{x\in\mathds{R}^d: (x-c)\TT M(x-c)\leq 1 \}$ and $\phi_k^\infty(M)$ is the polar function
\begin{equation}\label{polar-phi}
  \phi_k^\infty(M) = \inf_{V\succeq 0:\ \tr(MV)=1} \frac{1}{\Psi_k^{1/k}(V)} \,.
\end{equation}
\end{theorem}

The proof is given in Appendix. The polar function $\phi_k^\infty(\cdot)$ possesses the properties of what is called an information function in \citep[Chap.~5]{Pukelsheim93}; in particular, it is concave on the set of symmetric non-negative definite matrices. This duality property has the following consequence.

\begin{corollary} The determination of a covariance matrix $V_k^*$ that maximises $\Psi_k(V_\mu)$ with respect to $\mu\in\SM$ is equivalent to the determination of an ellipsoid $\SE(M_k^*,c_k^*)$ containing $\SX$, minimum in the sense that $M_k^*$ maximizes $\phi_k^\infty(M)$. The points of contact between $\SE(M_k^*,c_k^*)$ and $\SX$ form the support of $\mu_k^*$.
\end{corollary}

For any $V\succeq 0$, denote by $M_*(V)$ the matrix
\begin{equation}\label{M_*}
  M_*(V) = \frac{\nabla_{\Psi_k}[V]}{k\,\Psi_k(V)]} = \frac1k\, \nabla_{\log\Psi_k}[V]\,.
\end{equation}
Note that $M_*(V)\succeq 0$, see \cite[Lemma~7.5]{Pukelsheim93}, and that
$$
\tr[VM_*(V)]=1 \,,
$$
see the proof of Theorem~\ref{th:equivTh}.
The matrix $V\succeq 0$ maximises $\Psi_k(V)$ under the constraint $\tr(MV)=1$ for some $M\succeq 0$ if and only if $V[M_*(V)-M]=0$. Therefore, if $M$ is such that there exists $V_*=V_*(M)\succeq 0$ such that $M=M_*[V_*(M)]$, then $\phi_k^\infty(M)=\Psi_k^{-1/k}[V_*(M)]$. When $k<d$, the existence of such a $V_*$ is not ensured for all $M\succeq 0$, but happens when $M=M_k^*$ which maximises $\phi_k^\infty(M)$ under the constraint $\SX\in\SE(M,c)$. Moreover, in that case there exists a $\mu_k^*\in\SM$ such that $M_k^*=M_*(V_{\mu_k^*})$, and this $\mu_k^*$ maximises $\psi_k(\mu)$ with respect to $\mu\in\SM$.

Consider in particular the case $k=1$. Then, $M_*(V)=I_d/\tr(V)$ and $\phi_1^\infty(M)=\ml_{\min}(M)/2$. The matrix $M_k^*$ of the optimal ellipsoid $\SE(M_k^*,c_k^*)$ is proportional to the identity matrix and $\SE(M_k^*,c_k^*)$ is the ball of minimum-volume that encloses $\SX$.

When $k=2$ and $I_d \succeq (d-1)M/\tr(M)$, direct calculations show that $\phi_2^\infty(M)=\Psi_2^{-1/2}[V_*(M)]$, with
$$
V_*(M)=[I_d\,\tr(M)/(d-1) - M]\,[\tr^2(M)/(d-1)-\tr(M^2)]^{-1}\,;
$$
the optimal ellipsoid is then such that $\tr^2(M)/(d-1)-\tr(M^2)$ is maximised.


\subsubsection{Examples}\label{S:Examples}

\paragraph{Example 2}

Take $\SX=[0,1]^d$, $d\geq 1$ and denote by $v_i$, $i=1,\ldots,2^d$ the $2^d$ vertices of $\SX$. Consider $\mu^*=(1/2^d)\sum_{i=1}^{2^d} \delta_{v_i}$, with $\delta_v$ the Dirac delta measure at $v$. Then, $V_{\mu^*}=I_d/4$ and one can easily check that $\mu^*$ is $\psi_1$-optimal. Indeed, $E_{\mu^*}=\1b_d/2$, with $\1b_d$ the $d$-dimensional vector of ones, and
$\max_{x\in\SX} (x-\1b_d/2)\TT (2\,I_d) (x-\1b_d/2)= d/2 = \tr\{V_{\mu^*}\nabla_{\Psi_1}[V_{\mu^*}]\}$. From Remark~\ref{R:Radoslav}, the measure $\mu^*$ is $\psi_k$-optimal for all $k=1,\ldots,d$.

Note that the two-point measure $\mu_1^*=(1/2)[\delta_\0b + \delta_{\1b_d}]$ is such that $V_{\mu_1^*}=(\1b_d\,\1b_d\TT)/4$ and $\psi_1(\mu_1^*)=d/2=\psi_1(\mu^*)$, and is therefore $\psi_1$-optimal too. It is not $\psi_k$-optimal for $k>1$, since $\psi_k(\mu_1^*)=0$, $k>1$.

\paragraph{Example 3}

Take $\SX=\SB_d(\0b,\rho)$, the closed ball of $\mathds{R}^d$ centered at the origin $\0b$ with radius $\rho$. Let $\mu_0$ be the uniform measure on the sphere
$\SS_d(\0b,\rho)$ (the boundary of $\SB_d(\0b,\rho)$). Then, $V_{\mu_0}$ is proportional to the identity matrix $I_d$, and $\tr[V_{\mu_0}]=\rho^2$ implies that
$V_{\mu_0}=\rho^2 I_d/d$. Take $k=d$. We have $E_{\mu_0}=0$ and
$$
\max_{x\in\SX} (x-E_{\mu_0})\TT \nabla_{\Psi_d}[V_{\mu_0}](x-E_{\mu_0}) = \frac{(d+1) \rho^{2d}}{d^{d-1}d!} =\tr\{V_{\mu_0}\nabla_{\Psi_d}[V_{\mu_0}]\} \,,
$$
so that $\mu_0$ is $\psi_d$-optimal from \eqref{CNS}.

Let $\mu_d$ be the measure that allocates mass $1/(d+1)$ at each vertex of a $d$ regular simplex having its $d+1$ vertices on $\SS_d(\0b,\rho)$, with squared volume $\rho^{2d} (d+1)^{d+1}/[d^d (d!)^2]$. We also have $V_{\mu_d}=\rho^2 I_d/d$, so that $\mu_d$ is $\psi_d$-optimal too. In view of Remark~\ref{R:Radoslav}, $\mu_0$ and $\mu_d$ are $\psi_k$-optimal for all $k$ in $\{1,\ldots,d\}$.

Let now $\mu_k$ be the measure that allocates mass $1/(k+1)$ at each vertex of a $k$ regular simplex $\SP_k$, centered at the origin, with its vertices on $\SS_d(\0b,\rho)$. The squared volume of $\SP_k$ equals $\rho^{2k}\, (k+1)^{k+1}/[k^k (k!)^2]$.
Without any loss of generality, we can choose the orientation of the space so that $V_{\mu_k}$ is diagonal, with its first $k$ diagonal elements equal to $\rho^2/k$ and the other elements equal to zero. Note that $\psi_{k'}(\mu_k)=0$ for $k'>k$. Direct calculations based on \eqref{Psi_kE_k} give
$$
\psi_k(\mu_k)= \frac{k+1}{k!}\, \frac{\rho^{2k}}{k^k} \leq \psi_k(\mu_0) = \frac{k+1}{k!}\, {d \choose k}\, \frac{\rho^{2k}}{d^k} \,,
$$
with equality for $k=1$ and $k=d$, the inequality being strict otherwise.

%

\subsection{Optimal design in regression models}\label{S:design}

In this section we consider the case when $V=M^{-1}(\xi)$, where $M(\xi)$ is the information matrix
$$
M(\xi) = \int_{\ST} f(t)f\TT(t)\, \xi(\dd t)
$$
in a regression model $Y_j=\mt\TT f(t_j)+\mve_j$ with parameters $\mt\in\mathds{R}^d$, for a design measure $\xi\in\Xi$. Here $\Xi$ denotes the set of probability measures on a set $\ST$ such that $\{f(t): t\in\ST\}$ is compact, and $M^{-1}(\xi)$ is the (asymptotic) covariance matrix of an estimator $\mth$ of $\mt$ when the design variables $t$ are distributed according to $\xi$. The value $\psi_k(\mu)$ of Theorem~\ref{Prop:1} defines a measure of dispersion  for $\mth$, that depends on $\xi$ through $V_\mu=M^{-1}(\xi)$. The design problem we consider consists in choosing $\xi$ that minimises this dispersion, as measured by $\Psi_k[M^{-1}(\xi)]$, or equivalently that maximises $\Psi_k^{-1}[M^{-1}(\xi)]$.

\subsubsection{Properties}
It is customary in optimal design theory to maximise a concave and Loewner-increasing function of $M(\xi)$, see \cite[Chap.~5]{Pukelsheim93} for desirable properties of optimal design criteria. Here we have the following.

\begin{theorem}\label{Th:design} The functions $M \longrightarrow \Psi_k^{-1/k}(M^{-1})$, $k=1,\ldots,d$, are Loew\-ner-increasing, concave and differentiable on the set $\mathbb{M}^+$ of $d\times d$ symmetric positive-definite matrices. The functions $\Psi_k(\cdot)$ are also orthogonally invariant.
\end{theorem}

\begin{proof}
The property \eqref{EkEd-k} yields
\begin{equation}\label{design-criterion}
  \Psi_k^{-1/k}(M^{-1}) = \left(\frac{k+1}{k!}\right)^{-1/k}\, \frac{\det^{1/k}(M)}{\mathcal{E}_{d-k}^{1/k}(M)}
\end{equation}
which is a concave function of $M$, see Eq.~(10) of \cite[p.~116]{MarcusM64}. Since $\Psi_k(\cdot)$ is Loewner-increasing, see \cite{Lopez-FR-D98-MODA}, the function $M \longrightarrow \Psi_k^{-1/k}(M^{-1})$ is Loewner-increasing too. Its orthogonal invariance follows from the fact that it is defined in terms of the eigenvalues of $M$.
\end{proof}

Note that Theorems~\ref{Prop:1} and \ref{Th:design} imply that the functions
$M \longrightarrow - \log \Psi_k(M)$
and $M \longrightarrow \log \Psi_k(M^{-1})$ are convex for all $k=1,\ldots,d$, a question which was left open in \citep{Lopez-FR-D98-MODA}.

As a consequence of Theorem~\ref{Th:design}, we can derive a necessary and sufficient condition for a design measure $\xi_k^*$ to maximise  $\Psi_k^{-1/k}[M^{-1}(\xi)]$ with respect to $\xi\in\Xi$, for $k=1,\ldots,d$.

\begin{theorem}\label{th:equivTh2} The design measure $\xi_k^*$ such that $M(\xi_k^*)\in \mathbb{M}^+$ maximises $\tilde\psi_k(\xi)=\Psi_k^{-1/k}[M^{-1}(\xi)]$ with respect to $\xi\in\Xi$ if and only if
\begin{equation}\label{CNS-design-a}
  \max_{t\in\ST} f\TT(t)M^{-1}(\xi_k^*)\, \frac{\nabla_{\Psi_k}[M^{-1}(\xi_k^*)]}{\Psi_k[M^{-1}(\xi_k^*)]}\,M^{-1}(\xi_k^*)f(t) \leq k
\end{equation}
or, equivalently,
\begin{equation}\label{CNS-design}
  \max_{t\in\ST} \left\{ f\TT(t)M^{-1}(\xi_k^*)f(t) - f\TT(t)\frac{\nabla_{\Psi_{d-k}}[M(\xi_k^*)]}{\Psi_{d-k}[M(\xi_k^*)]}f(t) \right\} \leq d - k    \,.
\end{equation}
Moreover, there is equality in \eqref{CNS-design-a} and \eqref{CNS-design} for all $t$ in the support of $\xi_k^*$.
\end{theorem}

\begin{proof}
From \eqref{design-criterion}, the maximisation of $\tilde\psi_k(\xi)$ is equivalent to the maximisation of
$\tilde\phi_k(\xi)=\log\det[M(\xi)]-\log\Psi_{d-k}[M(\xi)]$.
The proof is similar to that of Theorem~\ref{th:equivTh} and is based on the following expressions for the directional derivatives of these two functionals
at $\xi$ in the direction $\nu\in\Xi$,
$$
F_{\tilde\psi_k}(\xi;\nu) = \tr\left( \frac1k\, M^{-1}(\xi)\, \frac{\nabla_{\Psi_k}[M^{-1}(\xi)]}{\Psi_k[M^{-1}(\xi)]}\,M^{-1}(\xi)\, [M(\nu)-M(\xi)] \right) $$
and
$$
F_{\tilde\phi_k}(\xi;\nu) = \tr\left( \left\{M^{-1}(\xi)-\frac{\nabla_{\Psi_{d-k}}[M(\xi)]}{\Psi_{d-k}[M(\xi)]} \right\}[M(\nu)-M(\xi)] \right) \,,
$$
and on the property $\tr\{M\nabla_{\Psi_{j}}[M]\}=j\, \Psi_{j}(M)$. 
\end{proof}

In particular, consider the following special cases for $k$ (note that $\Psi_0(M)=\mathcal{E}_0(M)=1$ for any $M$). 
\begin{eqnarray*}
&&k=d: \hspace{1.2cm} \tilde\psi_d(\xi) = \log\det[M(\xi)] \,, \\
&&k=d-1: \hspace{0.5cm} \tilde\psi_{d-1}(\xi) = \log\det[M(\xi)] - \log\tr[M(\xi)] - \log 2 \,,\\
&&k=d-2:  \hspace{0.5cm} \tilde\psi_{d-2}(\xi) = \log\det[M(\xi)]  \\
&& \hspace{4cm} - \log\left\{\tr^2[M(\xi)]-\tr[M^2(\xi)]\right\} - \log(3/4)\,. \\
\end{eqnarray*}
The necessary and sufficient condition \eqref{CNS-design} then takes the following form:
\begin{eqnarray*}
&&\hspace{-0.5cm} k=d: \hspace{1cm} \max_{t\in\ST} f\TT(t)M^{-1}(\xi_k^*)f(t) \leq d \,,\\
&&\hspace{-0.5cm} k=d-1: \hspace{0.3cm} \max_{t\in\ST} \left\{ f\TT(t)M^{-1}(\xi_k^*)f(t) - \frac{f\TT(t)f(t)}{\tr[M(\xi_{k}^*)]} \right\} \leq d -1\,, \\
&&\hspace{-0.5cm} k=d-2:  \hspace{0.3cm} \max_{t\in\ST} \left\{ f\TT(t)M^{-1}(\xi_k^*)f(t) \right. \\
&& \hspace{3cm} \left. - 2\,\frac{\tr[M(\xi_{k}^*)]f\TT(t)f(t)-f\TT(t)M(\xi_{k}^*)f(t)}{\tr^2[M(\xi_{k}^*)]-\tr[M^2(\xi_{k}^*)]} \right\} \leq d -2 \,.\\
\end{eqnarray*}
Also, for $k=1$ condition \eqref{CNS-design-a} gives
$$
\max_{t\in\ST} f\TT(t) \frac{M^{-2}(\xi_1^*)}{\tr[M^{-1}(\xi_1^*)]} f(t) \leq 1
$$
(which corresponds to $A$-optimal design), and for $k=2$
$$
\max_{t\in\ST} \frac{\tr[M^{-1}(\xi_2^*)] f\TT(t)M^{-2}(\xi_2^*)f(t) - f\TT(t)M^{-3}(\xi_2^*)f(t)}{\tr^2[M^{-1}(\xi_2^*)]-\tr[M^{-2}(\xi_2^*)]} \leq 1 \,.
$$

Finally, note that a duality theorem, in the spirit of Theorem~\ref{Th:duality}, can be formulated for the maximisation of $\Psi_k^{-1/k}[M^{-1}(\xi)]$; see \cite[Th.~7.12]{Pukelsheim93} for the general form a such duality properties in optimal experimental design.

\subsubsection{Examples}
\paragraph{Example 4} For the linear regression model on $\mt_0+\mt_1\,x$ on $[-1,1]$, the optimal design for $\tilde\psi_k(\cdot)$ with $k=d=2$ or $k=1$ is
$$
\xi_k^*=\left\{
\begin{array}{cc}
-1  & 1 \\
1/2 & 1/2
\end{array} \right\}\,,
$$
where the first line corresponds to support points and the second indicates their respective weights.

\paragraph{Example 5} For linear regression with the quadratic polynomial model $\mt_0+\mt_1\,t+\mt_2\,t^2$ on $[-1,1]$, the optimal designs for $\tilde\psi_k(\cdot)$ have the form
$$
\xi_k^*=\left\{
\begin{array}{ccc}
-1 & 0 & 1 \\
w_k & 1-2w_k & w_k
\end{array} \right\}\,,
$$
with $w_3 =  1/3$,  $w_2 = (\sqrt{33}-1)/16 \simeq 0.2965352$ and $w_1 = 1/4$. Define the efficiency $\mathrm{Eff}_k(\xi)$ of a design $\xi$ as
$$
\mathrm{Eff}_k(\xi) = \frac{\tilde\psi_k(\xi)}{\tilde\psi_k(\xi_k^*)} \,.
$$
Table~\ref{Tab:Ex5} gives the efficiencies $\mathrm{Eff}_k(\xi_j^*)$ for $j,k=1,\ldots,d=3$. The design $\xi_2^*$, optimal for $\tilde\psi_2(\cdot)$, appears to make a good compromise between $A$-optimality (which corresponds to $\tilde\psi_1(\cdot)$) and $D$-optimality (which corresponds to $\tilde\psi_3(\cdot)$).

\begin{table}[t]
\centering
\caption{\small Efficiencies $\mathrm{Eff}_k(\xi_j^*)$ for $j,k=1,\ldots,d$ in Example 5.}
{\small \begin{tabular}{lccc}
\hline
        & $\mathrm{Eff}_1$ & $\mathrm{Eff}_2$ & $\mathrm{Eff}_3$ \\
\hline
$\xi_1^*$ & 1 & 0.9770 & 0.9449 \\
$\xi_2^*$ & 0.9654 & 1 & 0.9886 \\
$\xi_3^*$ & 0.8889 & 0.9848 & 1 \\
        \hline
\end{tabular}}

\label{Tab:Ex5}
\end{table}

\paragraph{Example 6} For linear regression with the cubic polynomial model $\mt_0+\mt_1\,t+\mt_2\,t^2+\mt_3\,t^3$ on $[-1,1]$, the optimal designs for $\tilde\psi_k(\cdot)$ have the form
$$
\xi_k^*=\left\{
\begin{array}{cccc}
-1 & -z_k & z_k & 1 \\
w_k & 1/2-w_k & 1/2-w_k & w_k
\end{array} \right\}\,,
$$
where
$$
  \begin{array}{ll}
    z_4 = 1/\sqrt{5} \simeq 0.4472136\,, & w_4 =0.25\,, \\
    z_3 \simeq 0.4350486\,, & w_3 \simeq  0.2149859\,,  \\
    z_2 \simeq 0.4240013\,, & w_2 \simeq  0.1730987\,, \\
    z_1 = \sqrt{3\sqrt{7}-6}/3 \simeq 0.4639509\,, & w_1 = (4-\sqrt{7})/9 \simeq 0.1504721\,,
  \end{array}
$$
with $z_3$ satisfying the equation $2z^6-3z^5-45z^4+6z^3-4z^2-15z+3=0$ and \\
$$
w_3=
\,{\frac {5\,{z}^{6}+5\,{z}^{4}+5\,{z}^{2}+1-
\sqrt {{z}^{12}+2\,{z
}^{10}+3\,{z}^{8}+60\,{z}^{6}+59\,{z}^{4}+58\,{z}^{2}+73}}
{12({z}^{6}+{
z}^{4}+{z}^{2}-3)}}\,,
$$
with $z=z_3$. For $k=d-2=2$, the numbers $z_2$ and $w_2$ are too difficult to express analytically. Table~\ref{Tab:Ex6} gives the efficiencies $\mathrm{Eff}_k(\xi_j^*)$ for $j,k=1,\ldots,d$. Here again the design $\xi_2^*$ appears to make a good compromise: it maximises the minimum efficiency $\min_k \mathrm{Eff}_f(\cdot)$ among the designs considered.

\begin{table}[t]
\centering
\caption{\small Efficiencies $\mathrm{Eff}_k(\xi_j^*)$ for $j,k=1,\ldots,d$ in Example 6.}
{\small \begin{tabular}{lcccc}
\hline
        & $\mathrm{Eff}_1$ & $\mathrm{Eff}_2$ & $\mathrm{Eff}_3$ & $\mathrm{Eff}_4$ \\
\hline
$\xi_1^*$ & 1 & 0.9785 & 0.9478 & 0.9166 \\
$\xi_2^*$ & 0.9694 & 1 & 0.9804 & 0.9499 \\
$\xi_3^*$ & 0.9180 & 0.9753 & 1 & 0.9897 \\
$\xi_4^*$ & 0.8527 & 0.9213 & 0.9872 & 1 \\
        \hline
\end{tabular}}
\label{Tab:Ex6}
\end{table}

\section*{Appendix}

\paragraph{Shift-invariance and positive homogeneity}

Denote by $\SM$ the set of probability measures defined on the Borel subsets of $\SX$, a compact subset of $\mathds{R}^d$. For any $\mu\in\SM$, any $\mt\in \mathds{R}^d$ and any $\ml\in\mathds{R}^+$, respectively denote by $T_{-\mt}[\mu]$ and $H_{\ml^{-1}}[\mu]$ the measures defined by:
$$
\mbox{ for any } \mu\mbox{-measurable } \SA\subseteq\SX\,, \ T_{-\mt}[\mu](\SA+\mt)=\mu(\SA)\,, \ H_{\ml^{-1}}[\mu](\ml\SA)=\mu(\SA) \,,
$$
where $\SA+\mt=\{x+\mt: x\in\SA\}$ and $\ml\SA=\{\ml\,x: x\in\SA\}$.
The shift-invariance of $\phi(\cdot)$ then means that $\phi(T_{-\mt}[\mu])=\phi(\mu)$ for any $\mu\in\SM$ and any $\mt\in\mathds{R}^d$,
positive homogeneity of degree $q$ means that $\phi(H_{\ml^{-1}}[\mu])=\ml^q\, \phi(\mu)$ for any $\mu\in\SM$ and any $\ml\in\mathds{R}^+$.
\fin

\paragraph{The variance is the only concave central moment}

For $q\neq 2$, the $q$-th central moment
$\Delta_q(\mu)=\int |x-E_\mu|^q\,\mu(\dd x)$
is shift-invariant and homogeneous of degree $q$, but it is not concave on $\SM$. Indeed, consider for instance the two-point probability measures
$$
\mu_1=\left\{ \begin{array}{cc}
0 & 1 \\
1/2 & 1/2
\end{array}\right\}
\mbox{ and }
\mu_2=\left\{ \begin{array}{cc}
0 & 101 \\
w & 1-w
\end{array}\right\}\,,
$$
where the first line denotes the support points and the second one their respective weights. Then, for
$$
w=1-\frac{1}{404}\, \frac{201^{q-1}-202q+405}{201^{q-1}-101q+102}
$$
one has
$\mp^2 \Delta_q[(1-\ma)\mu_1+\ma\mu_2]/\mp\ma^2\big|_{\ma=0} \geq 0$ for all $q \geq 1.84$, the equality being obtained at $q=2$ only. Counterexamples are  easily constructed for values of $q$ smaller than 1.84.
\fin

\paragraph{Proof of Lemma~\ref{L:1}} 
We have
$$
\Ex \left \{ \det\left [\sum_{i=1}^{k+1} z_i z_i\TT \right] \right\} = (k+1)!\, \det\left[
  \begin{array}{cc}
    \Ex(x_1 x_1\TT) & E_\mu \\
    E_\mu\TT & 1 \\
  \end{array}
\right] = (k+1)! \det[V_\mu]\,,
$$
see for instance \cite[Theorem~1]{Pa98}. 
\fin

\paragraph{Proof of Lemma~\ref{L:2}} 
Take any vector $z$ of the same dimension as $x$. Then $z\TT V_\mu z=\var_\mu(z\TT x)$, which is a concave functional of $\mu$, see Section~\ref{S:intro}. This implies that
$z\TT V_{(1-\ma)\mu_1+\ma\mu_2} z = \var_{(1-\ma)\mu_1+\ma\mu_2}(z\TT x) \geq (1-\ma)\var_{\mu_1}(z\TT x)+\ma \var_{\mu_2}(z\TT x) = (1-\ma) z\TT V_{\mu_1}z +\ma z\TT V_{\mu_2}z$,
for any $\mu_1$, $\mu_2$ in $\SM$ and any $\ma\in(0,1)$ (see Section~\ref{S:intro} for the concavity of $\var_\mu$).
Since $z$ is arbitrary, this implies \eqref{conc0}.
\fin

\paragraph{Proof of Theorem~\ref{Th:empirical}} 
The estimate \eqref{widehat-psi} forms a U-statistics for the estimation of $\psi_k(\mu)$ and is thus unbiased and has minimum variance, see, e.g., \cite[Chap.~5]{Serfling80}.
We only need to show that it can be written as \eqref{widehat-psi-b}.

We can write
\begin{eqnarray*}
({\widehat\psi}_k)_n  &=& {n \choose k+1}^{-1} \\
&& \times \sum_{j_1<j_2<\cdots<j_{k+1}} \frac{1}{(k!)^2} \, \sum_{i_1<i_2<\cdots<i_k} {\det}^2
\left[
  \begin{array}{ccc}
    \{x_{j_1}\}_{i_1} & \cdots & \{x_{j_{k+1}}\}_{i_1} \\
    \vdots & \vdots & \vdots \\
    \{x_{j_1}\}_{i_k} & \cdots & \{x_{j_{k+1}}\}_{i_k} \\
    1 & \cdots & 1
  \end{array}
\right] \,, \\
&=& {n \choose k+1}^{-1}\, \frac{1}{(k!)^2} \, \sum_{i_1<i_2<\cdots<i_k} \det \left( \sum_{j=1}^n
\{z_j\}_{i_1,\ldots,i_k} \{z_j\}\TT_{i_1,\ldots,i_k}\right) \,,
\end{eqnarray*}
where we have used Binet-Cauchy formula and where $\{z_j\}_{i_1,\ldots,i_k}$ denotes the $k+1$ dimensional vector with components $\{x_j\}_{i_\ell}$, $\ell=1,\ldots,k$, and 1. This gives
\begin{eqnarray*}
({\widehat\psi}_k)_n  &=& {n \choose k+1}^{-1}\, \frac{n^{k+1}}{(k!)^2} \, \sum_{i_1<i_2<\cdots<i_k} \det \left(\frac1n\, \sum_{j=1}^n
 \{z_j\}_{i_1,\ldots,i_k} \{z_j\}\TT_{i_1,\ldots,i_k}\right) \,, \\
 &=& {n \choose k+1}^{-1}\, \frac{n^{k+1}}{(k!)^2} \\
 && \hspace{-0.5cm} \times \sum_{i_1<i_2<\cdots<i_k} \det
 \left[
   \begin{array}{cc}
     (1/n) \{\sum_{j=1}^n x_jx_j\TT\}_{(i_1,\ldots,i_k)\times(i_1,\ldots,i_k)} & \{\widehat x_n\}_{i_1,\ldots,i_k} \\
     \{\widehat x_n\}\TT_{i_1,\ldots,i_k} & 1 \\
   \end{array}
 \right]
  \,, \\
  &=& {n \choose k+1}^{-1}\, \frac{n^{k+1}}{(k!)^2} \, \sum_{i_1<i_2<\cdots<i_k} \det
 \left[\frac{n-1}{n}\, \{\widehat V_n\}_{(i_1,\ldots,i_k)\times(i_1,\ldots,i_k)} \right]
  \,, \\
\end{eqnarray*}
and thus \eqref{widehat-psi-b}.
\fin

\paragraph{Proof of Theorem~\ref{Th:duality}} \mbox{} \\
($i$) The fact that $\max_{\mu\in\SM} \Psi_k^{1/k}(V_\mu) \geq \min_{M,c:\ \SX\subset\SE(M,c)} 1/\phi_k^\infty(M)$ is a consequence of Theorem~\ref{th:equivTh}. Indeed, the measure $\mu_k^*$ maximises $\Psi_k^{1/k}(V_\mu)$ if and only if
\begin{equation}\label{ET}
  (x-E_{\mu_k^*})\TT M_*(V_{\mu_k^*})(x-E_{\mu_k^*}) \leq 1 \ \mbox{ for all $x$ in $\SX$.}
\end{equation}
Denote $M_k^*=M_*(V_{\mu_k^*})$, $c_k^*=E_{\mu_k^*}$, and consider the Lagrangian $L(V,\ma;M)$ for the maximisation of $(1/k) \log\Psi_k(V)$ with respect to $V\succeq 0$ under the constraint $\tr(MV)=1$:
$
L(V,\ma;M) = (1/k) \log\Psi_k(V) - \ma[\tr(MV)-1] \,.
$
We have
$$
\frac{\mp L(V,1;M_k^*)}{\mp V}\bigg|_{V=V_{\mu_k^*}} = M_k^* - M_k^*= 0
$$
and $\tr(M_k^*V_{\mu_k^*})=1$, with $V_{\mu_k^*}\succeq 0$. Therefore, $V_{\mu_k^*}$ maximises $\Psi_k(V)$ under the constraint $\tr(M_k^*V)=1$, and, moreover, $\SX\subset\SE(M_k^*,c_k^*)$ from \eqref{ET}. This implies
\begin{eqnarray*}
\lefteqn{ \Psi_k^{1/k}(V_{\mu_k^*}) = \max_{V\succeq 0:\ \tr(M_k^*V)=1} \Psi_k^{1/k}(V) } \\
&&\geq \min_{M,c:\ \SX\subset\SE(M,c)} \ \ \max_{V\succeq 0:\ \tr(MV)=1} \Psi_k^{1/k}(V)  = \min_{M,c:\ \SX\subset\SE(M,c)} \frac{1}{\phi_k^\infty(M)} \,.
\end{eqnarray*}

\vsp
\noindent($ii$) We prove now that $\min_{M,c:\ \SX\subset\SE(M,c)} 1/\phi_k^\infty(M) \geq \max_{\mu\in\SM} \Psi_k^{1/k}(V_\mu)$. Note that we do not have an explicit form for $\phi_k^\infty(M)$ and that the infimum in \eqref{polar-phi} can be attained at a singular $V$, not necessarily unique, so that we cannot differentiate $\phi_k^\infty(M)$. Also note that compared to the developments in \citep[Chap.~7]{Pukelsheim93}, here we consider covariance matrices instead of moment matrices.

Consider the maximisation of $\log \phi_k^\infty(M)$ with respect to $M$ and $c$ such that $\SX\subset\SE(M,c)$, with Lagrangian
$$
L(M,c,\beta)=\log \phi_k^\infty(M) + \sum_{x\in\SX} \beta_x[1-(x-c)\TT M(x-c)]\,, \ \beta_x \geq 0 \mbox{ for all $x$ in $\SX$.}
$$
For the sake of simplicity we consider here $\SX$ to be finite, but $\beta$ may denote any positive measure on $\SX$ otherwise.
Denote the optimum by
$$
T^* = \max_{M,c:\ \SX\subset\SE(M,c)} \log \phi_k^\infty(M) \,.
$$
It satisfies
$$
T^* = \max_{M,c} \min_{\beta\geq 0} L(M,c,\beta) \leq  \min_{\beta\geq 0} \max_{M,c} L(M,c,\beta)
$$
and $\max_{M,c} L(M,c,\beta)$ is attained for any $c$ such that
$$
Mc=M \sum_{x\in\SX} \beta_x\,x/(\sum_{x\in\SX} \beta_x)\,,
$$
that is, in particular for
$$
c^*= \frac{\sum_{x\in\SX} \beta_x\,x}{\sum_{x\in\SX} \beta_x} \,,
$$
and for $M^*$ such that $0\in\partial_M L(M,c^*,\beta)\big|_{M=M^*}$, the subdifferential of $L(M,c^*,\beta)$ with respect to $M$ at $M^*$. This condition can be written as
$$
\sum_{x\in\SX} \beta_x\,(x-c^*)(x-c^*)\TT = \tilde V \in \partial \log \phi_k^\infty(M)\big|_{M=M^*} \,,
$$
with $\partial \log \phi_k^\infty(M)$ the subdifferential of $\log \phi_k^\infty(M)$,
$$
\partial \log \phi_k^\infty(M) = \{ V \succeq 0: \ \Psi_k^{1/k}(V)\phi_k^\infty(M)=\tr(MV)=1 \}\,,
$$
see \cite[Th.~7.9]{Pukelsheim93}. Since $\tr(MV)=1$ for all $V\in\partial \log \phi_k^\infty(M)$, $\tr(M^*\tilde V)=1$ and thus
$
\sum_{x\in\SX} \beta_x\, (x-c^*)\TT M^*(x-c^*) = 1 \,.
$
Also, $\Psi_k^{1/k}(\tilde V) = 1/\phi_k^\infty(M^*)$, which gives
$$
L(M^*,c^*,\beta) = - \log \Psi_k^{1/k}\left[\sum_{x\in\SX} \beta_x\,(x-c^*)(x-c^*)\TT\right] + \sum_{x\in\SX} \beta_x -1 \,.
$$
We obtain finally
\begin{eqnarray*}
\lefteqn{ \min_{\beta\geq 0} L(M^*,c^*,\beta) } \\
&& = \min_{\mg>0,\ \ma \geq 0} \left\{ - \log \Psi_k^{1/k}\left[\sum_{x\in\SX} \ma_x\,(x-c^*)(x-c^*)\TT\right]  + \mg - \log(\mg) -1 \right\} \,, \\
&& = \min_{\ma \geq 0} - \log \Psi_k^{1/k}\left[\sum_{x\in\SX} \ma_x\,(x-c^*)(x-c^*)\TT\right] = - \log\Psi_k^{1/k}(V_k^*) \,,
\end{eqnarray*}
where we have denoted $\mg=\sum_{x\in\SX} \beta_x$ and $\ma_x=\beta_x/\mg$ for all $x$. Therefore
$T^* \leq - \log\Psi_k^{1/k}(V_k^*)$, that is,
$\log\left[\min_{M,c:\ \SX\subset\SE(M,c)} 1/\phi_k^\infty(M)\right] \geq  \log\Psi_k^{1/k}(V_k^*)$.
\fin

\section*{Acknowledgments}
The work of the first author was partly supported by the ANR project 2011-IS01-001-01 DESIRE (DESIgns for spatial Random fiElds).

\bibliographystyle{elsart-harv}

\bibliography{PWZ_entropy_and_distances}

\end{document}